\documentclass[12pt,a4paper]{article}
\usepackage{amsfonts}
\usepackage{amsfonts}
\usepackage{amsfonts}
\usepackage{amsfonts}
\usepackage{amsfonts}
\usepackage{amsfonts}
\usepackage{amsfonts}
\usepackage{amsfonts}
\usepackage{amsfonts}
\usepackage{amsfonts}
\usepackage{amsfonts}
\usepackage{amsfonts}
\usepackage{mathrsfs}
\usepackage{mathrsfs}
\usepackage{mathrsfs}
\usepackage{mathrsfs}
\usepackage{mathrsfs}
\usepackage{mathrsfs}
\usepackage{mathrsfs}
\usepackage{mathrsfs}
\usepackage{mathrsfs}
\usepackage{amsfonts}
\usepackage{amsfonts}
\usepackage{amsfonts}
\usepackage{amsfonts}
\usepackage{amsfonts}
\usepackage{amsfonts}
\usepackage{mathrsfs}
\usepackage{amsfonts}
\usepackage{amsfonts}
\usepackage{amsfonts}
\usepackage{amsfonts}
\usepackage{amsfonts}
\usepackage{amsfonts}
\usepackage{amsfonts}
\usepackage{amsfonts}
\usepackage{amsfonts}
\usepackage{amsfonts}
\usepackage{amsfonts}
\usepackage{amsfonts}
\usepackage{mathrsfs}
\usepackage{amsfonts}
\usepackage{mathrsfs}
\usepackage{amsfonts}
\usepackage{amsfonts}
\usepackage{amsfonts}
\usepackage{mathrsfs}
\usepackage{amsfonts}
\usepackage{amsfonts}
\usepackage{amsfonts}
\usepackage{amsfonts}
\usepackage{amsmath}
\usepackage{mathrsfs}
\usepackage{amsfonts}
\usepackage{amsfonts}
\usepackage{amsfonts}
\usepackage{amsfonts}
\usepackage{amsfonts}
\usepackage{amsfonts}
\usepackage{amsfonts}
\usepackage{amsfonts}
\usepackage{amsfonts}
\usepackage{amsfonts}
\usepackage{amsfonts}
\usepackage{mathrsfs}
\usepackage{amsfonts}
\usepackage{amsfonts}
\usepackage{amsfonts}
\usepackage{amsfonts}
\usepackage{amsmath}
\usepackage{mathrsfs}
\usepackage{mathrsfs}
\usepackage{mathrsfs}
\usepackage{mathrsfs}
\usepackage{mathrsfs}
\usepackage{mathrsfs}
\usepackage{mathrsfs}
\usepackage{mathrsfs}
\usepackage{amsfonts}
\usepackage{amsfonts}
\usepackage{amsfonts}
\usepackage{amsfonts}
\usepackage{amsfonts}
\usepackage{amsfonts}
\usepackage{amsfonts}
\usepackage{amsfonts}
\usepackage{amsfonts}
\usepackage{amsfonts}
\usepackage{amsfonts}
\usepackage{amsfonts}
\usepackage{mathrsfs}
\usepackage{amsfonts}
\usepackage{mathrsfs}
\usepackage{amsfonts}
\usepackage{amsfonts}
\usepackage{amsfonts}
\usepackage{amsfonts}
\usepackage{amsfonts}
\usepackage{amsfonts}
\usepackage{amsfonts}
\usepackage{amsfonts}
\usepackage{amsfonts}
\usepackage{amsfonts}
\usepackage{amsfonts}
\usepackage{amsfonts}
\usepackage{amsfonts}
\usepackage{amsfonts}
\usepackage{amsfonts}
\usepackage{amsfonts}
\usepackage{amsfonts}
\usepackage{amsfonts}
\usepackage{mathrsfs}
\usepackage{mathrsfs}
\usepackage{mathrsfs}
\usepackage{mathrsfs}
\usepackage{amsfonts}
\usepackage{amsfonts}
\usepackage{amsfonts}
\usepackage{amsfonts}
\usepackage{amsfonts}
\usepackage{amsfonts}
\usepackage{amsfonts}
\usepackage{amsfonts}
\usepackage{amsfonts}
\usepackage{amsfonts}
\usepackage{amsfonts}
\usepackage{amsfonts}
\usepackage{amsfonts}
\usepackage{amsfonts}
\usepackage{amsfonts}
\usepackage{amsfonts}
\usepackage{amsfonts}
\usepackage{mathrsfs}
\usepackage{mathrsfs}
\usepackage{amsfonts}
\usepackage{amsfonts}
\usepackage{amsfonts}
\usepackage{amsfonts}
\usepackage{amsfonts}
\usepackage{amsfonts}
\usepackage{amsfonts}
\usepackage{amsfonts}
\usepackage{amsfonts}
\usepackage{amsfonts}
\usepackage{amsfonts}
\usepackage{amsfonts}
\usepackage{amsfonts}
\usepackage{mathrsfs}
\usepackage{mathrsfs}
\usepackage{amsfonts}
\usepackage{amsfonts}
\usepackage{amsfonts}
\usepackage{amsfonts}
\usepackage{amsfonts}
\usepackage{amsfonts}
\usepackage{mathrsfs}
\usepackage{mathrsfs}
\usepackage{amsfonts}
\usepackage{amsfonts}
\usepackage{amsfonts}
\usepackage{amsfonts}
\usepackage{amsfonts}
\usepackage{amsfonts}
\usepackage{amsfonts}
\usepackage{amsfonts}
\usepackage{amsfonts}
\usepackage{amsfonts}
\usepackage{amsfonts}
\usepackage{amsfonts}
\usepackage{amsfonts}
\usepackage{amsfonts}
\usepackage{mathrsfs}
\usepackage{amsfonts}
\usepackage{amsfonts}
\usepackage{amsfonts}
\usepackage{amsfonts}
\usepackage{amsfonts}
\usepackage{amsfonts}
\usepackage{amsfonts}
\usepackage{amsfonts}
\usepackage{amsfonts}
\usepackage{amsfonts}
\usepackage{amsfonts}
\usepackage{amsfonts}
\usepackage{amsfonts}
\usepackage{amsfonts}
\usepackage{amsfonts}
\usepackage{amsfonts}
\usepackage{amsmath}
\usepackage{mathrsfs}
\usepackage{amsfonts}
\usepackage{amsfonts}
\usepackage{amsfonts}
\usepackage{amsfonts}
\usepackage{amsmath}
\usepackage{mathrsfs}
\usepackage{amsfonts}
\usepackage{mathrsfs}
\usepackage{mathrsfs}
\usepackage{mathrsfs}
\usepackage{mathrsfs}
\usepackage{amsmath}
\usepackage{mathrsfs}
\usepackage{amsfonts}
\usepackage{color}
\usepackage{mathrsfs}

\usepackage{stmaryrd}
\usepackage{amsmath}
\usepackage{amssymb}

\usepackage{bbm}
\usepackage{mathrsfs}
\usepackage{graphicx}

\usepackage{enumerate}
\usepackage{theorem}

\makeatletter
\def\tank#1{\protected@xdef\@thanks{\@thanks
        \protect\footnotetext[0]{#1}}}
\def\bigfoot{

    \@footnotetext}
\makeatother

\newcommand{\ea}{\end{array}}
\newtheorem{theorem}{Theorem}[section]

\newtheorem{claim}{Claim}[section]

\newtheorem{lemma}{Lemma}[section]
\newtheorem{definition}{Definition}[section]

{\theorembodyfont{\rmfamily}
}

\newenvironment{proof}{Proof.}

\def \eref#1{\hbox{(\ref{#1})}}

\def\ds{d s}

\begin{document}
\title{\Large \bf Stochastic Porous Media Equation on General
Measure Spaces with Increasing Lipschitz Nonlinearties
\thanks{Research is supported by the DFG' through CRC 701, the
National Natural Science Foundation of China (No.11271169) and the
Project Funded by the Priority Academic Program Development of
Jiangsu Higher Education Institutions.} }

\author{{Michael R\"{o}ckner}$^a$\footnote{E-mail:roeckner@math.uni-bielefeld.de}~~~
{Weina Wu}$^b$\footnote{E-mail:wuweinaforever@163.com}~~~ {Yingchao
Xie}$^c$\footnote{E-mail:ycxie@jsnu.edu.cn}
\\
 \small  a. Faculty of Mathematics, University of Bielefeld, D-33501 Bielefeld, Germany.\\
 \small  b. School of Mathematical Sciences, Nankai University, Tianjin 300071, China. \\
 \small  c. School of Mathematics and Statistics, Jiangsu Normal University, Xuzhou 221116, China.}\,
\date{}
\maketitle

\begin{center}
\begin{minipage}{130mm}
{\bf Abstract.} We prove the existence and uniqueness of
probabilistically strong solutions to stochastic porous media
equations driven by time-dependent multiplicative noise on a general
measure space $(E, \mathscr{B}, \mu)$, and the Laplacian replaced
by a negative definite self-adjoint operator $L$. In the case of
Lipschitz nonlinearities $\Psi$, we in particular generalize
previous results for open $E\subset \mathbb{R}^d$ and
$L\!\!=$Laplacian to fractional Laplacians. We also generalize known
results on general measure spaces, where we succeeded in dropping
the transience assumption on $L$, in extending the set of allowed
initial data and in avoiding the restriction to superlinear behavior
of $\Psi$ at infinity for $L^2(\mu)$-initial data.

\vspace{3mm} {\bf Keywords:} Wiener process; Porous media equation;
Sub-Markovian contractive semigroup.

\end{minipage}
\end{center}

\section{Introduction}
\setcounter{equation}{0}
 \setcounter{definition}{0}

In this paper, we consider stochastic porous media equations (SPMEs)
of the following type:

\begin{equation} \label{eq:1}
\left\{ \begin{aligned}
&dX(t)-L\Psi(X(t))dt=B(t,X(t))dW(t),\ \text{in}\ [0,T]\times E,\\
&X(0)=x \text{~on~} E~ (\text{with~} x\in F^*_{1,2} \text{~or~}
L^2(\mu)),
\end{aligned} \right.
\end{equation}
where $T\in(0,\infty)$ is fixed, $L$ is the negative definite self-adjoint generator of a
sub-Markovian strongly continuous contraction semigroup
$(P_t)_{t\geq0}$ on $L^2(\mu):=L^2(E, \mathscr{B}, \mu)$, and
$(E, \mathscr{B}, \mu)$ is a standard measurable space (\cite{P67}) with a $\sigma$-finite measure $\mu$.
$\Psi(\cdot)\!:\! \mathbb{R}\!\rightarrow\! \mathbb{R}$ is a
monotonically nondecreasing Lipschitz continuous function, $B$ is a
progressively measurable process in the space of Hilbert-Schmidt
operator from $L^2(\mu)$ to $F^*_{1,2}$, $W(t)$ is an
$L^2(\mu)$-valued cylindrical $\mathscr{F}_t$-adapted Wiener process
on a probability space $(\Omega, \mathscr{F}, \mathbb{P})$ with
normal filtration $(\mathscr{F}_t)_{t\geq0}$. For the definition of
the Hilbert space $F^*_{1,2}$ and the precise conditions on $B$ we
refer to the next section.

In the special case when $E=\mathbb{R}^d$, $L$ is equal to the
Laplace operator $\Delta$ and $B$ is time-independent linear
multiplicative, equation \eref{eq:1} was recently analyzed in
\cite{BRR}. The aim of this paper is to prove analogous results as
in \cite{BRR} for the general case. The above framework is inspired
by the work of Fukushima and Kaneko \cite{FK} (see also \cite{HK}).

The main motivation for this generality is that we would like to
cover fractional powers of the Laplacian, i.e.,
$\!L\!=\!-(-\Delta)^\alpha,~\alpha\in(0,1)$, generalized
Schr\"{o}dinger operators, i.e.,
$\!L\!=\!\Delta+2\frac{\nabla\rho}{\rho}\cdot\nabla$, and Laplacians
on fractals (see Section 4 below).

Recently, there has been much work on stochastic versions of the
porous media equations. Based on the variational approach and
monotonicity assumptions on the coefficients, \cite{RRW} presents a
generalization of Krylov-Rozovskii's result \cite{KR} on the
existence and uniqueness of solutions to monotone stochastic
differential equations, which applies to a large class of stochastic
porous media equations. It should be said that in \cite{RRW} (see
also \cite{RW}), $\Psi$ is assumed to be continuous such that
$r\Psi(r)\!\rightarrow\!\infty$ as $r\!\rightarrow\!\infty$. In this
paper we show that for Lipschitz continuous $\Psi$ this condition
can be dropped for initial data in $L^2(\mu)$, extending the
corresponding result from \cite{BRR} to general operators $L$ as
above. We would also like to emphasize that in contrast to
\cite{RRW, RW}, in this paper, we do not assume that $L$ is the
generator of a transient Dirichlet form on
$L^2(E,\mathscr{B},\mu)$. In our case we can drop the transience
assumption. In particular, in contrast to \cite{RRW} (and
\cite{RW}), we do not need any restriction on $d$ when
$E=\mathbb{R}^d$ and $L=-(-\Delta)^\alpha,~\alpha\in(0,1]$. For more
references on stochastic porous media equations we refer to
\cite{BDR1}. In addition, we work in the state space $F^*_{1,2}$
which is larger than the state space $\mathscr{F}_e^*$ considered in
\cite{RRW}, hence we can allow more general initial conditions (as
done in \cite{RW} under assumptions much stronger than transience).

Section 4 of \cite{BRR} deals with the case where $\Psi$ is a
maximal monotone multivalued function with at most polynomial
growth. However, due to the multiplier problem, the existence is
obtained for $d\geq3$ only. We plan to extend also this result to
our more general equation \eref{eq:1}. This will be the subject of
our future work.

The paper is organized as follows: in Section 2, we recall some
notions concerning sub-Markovian semi-groups and introduce a
suitable Gelfand triple. Section 3 is devoted to verify the
existence and uniqueness of strong solutions to \eref{eq:1}. Note
that the Riesz isomorphism $1-L$, through which we identify
$H:=F^*_{1,2}$ and $H^*:= F_{1,2}$, plays an essential role in the
proof. In Section 4, we will apply our results to a number of
examples.

\section{Preliminaries}\label{sec.prelim}
\setcounter{equation}{0}
 \setcounter{definition}{0}

First of all, let us recall some basic definitions and spaces which
will be used throughout the paper (see \cite{Fukushima, FK, HK}).

Let $(E, \mathscr{B},\mu)$ be a $\sigma$-finite measure space.
Let $\{P_t\}_{t\geq0}$ be a strongly continuous contraction
sub-Markovian semigroup on $L^2(\mu)$ with negative definite
self-adjoint generator $(L, D(L))$.

The gamma-transform $V_r (r>0)$ of $\{P_t\}_{t\geq0}$ is defined by
$$V_r=\Gamma(\frac{r}{2})^{-1}\int_0^\infty s^{\frac{r}{2}-1}e^{-s}P_sds.$$
In this paper, we consider the Hilbert space $(F_{1,2},
\|\cdot\|_{F_{1,2}})$ defined by
$$F_{1,2}=V_1(L^2(\mu)),\ \text{with norm}\ \|u\|_{F_{1,2}}=|f|_2\ \ \text{for} \ \ u=V_1f,\ \ f\in L^2(\mu),$$
where the norm $|\cdot|_2$ is defined as
$|f|_2=(\int_E|f|^2d\mu)^{\frac{1}{2}}$. In particular,
$$V_1=(1-L)^{-\frac{1}{2}},\ \text{so that}\ \|u\|_{F_{1,2}}=|V_1^{-1}u|_2=|(1-L)^{\frac{1}{2}}u|_2.$$
The dual space of $F_{1,2}$ is denoted by $F^*_{1,2}$.

Let $H$ be a separable Hilbert space with inner product
$\langle\cdot, \cdot\rangle_H$ and $H^*$ its dual. Let $V$ be a
reflexive Banach space, such that $V\subset H$ continuously and
densely. Then for its dual space $V^*$ it follows that $H^*\subset
V^*$ continuously and densely. Identifying $H$ and $H^*$ via the
Riesz isomorphism we have that
\begin{eqnarray*}
V\subset H\subset V^*
\end{eqnarray*}
continuously and densely and if $_{V^*}\langle\cdot, \cdot\rangle_V$
denotes the dualization between $V^*$ and $V$ (i.e. $_{V^*}\langle
z,v \rangle_V:=z(v)$ for $z\in V^*$, $v\in V$), it follows that
\begin{eqnarray*}
_{V^*}\langle z,v \rangle_V=\langle z,v\rangle_H,\ \text{for\ all}\
z\in H,\ v\in V.
\end{eqnarray*}
$(V,H,V^*)$ is called a Gelfand triple.

In the following, we concentrate on finding a suitable Gelfand
triple $V\subset H\equiv H^*\subset V^*$ with $H:=F^*_{1,2}$. Let
$_{F_{1,2}}\langle\cdot,\ \cdot\rangle_{F^*_{1,2}}$ denote the
duality between $F_{1,2}$ and $F_{1,2}^*$. Define $(1-L):
F_{1,2}\rightarrow F^*_{1,2}$ as follows, given $u\in F_{1,2}$,
\begin{eqnarray}\label{eqnarray1}
_{F^*_{1,2}}\langle (1-L)u,~ v\rangle_{F_{1,2}}:=\int_E
(1-L)^{\frac{1}{2}}u\cdot(1-L)^{\frac{1}{2}}v~ d\mu \
\text{~for~all~} v\in F_{1,2}.
\end{eqnarray}
To show that $(1-L): F_{1,2}\rightarrow F^*_{1,2}$ is well-defined,
we have to prove that the right-hand side of \eref{eqnarray1}
defines a linear continuous function on $v\in F_{1,2}$ with respect
to $\|\cdot\|_{F_{1,2}}$. But for $u\in F_{1,2}$, we have for all
$v\in F_{1,2}$,
\begin{eqnarray*}
\big|_{F^*_{1,2}}\langle (1-L)u,~
v\rangle_{F_{1,2}}\big|&&\!\!\!\!\!\!\!\!=\Big|\int_E
(1-L)^{\frac{1}{2}}u\cdot(1-L)^{\frac{1}{2}}v~ d\mu\Big|\\
&&\!\!\!\!\!\!\!\!=\Big|\big\langle (1-L)^{\frac{1}{2}}u,~ (1-L)^{\frac{1}{2}}v\big\rangle_2\Big|\\
&&\!\!\!\!\!\!\!\!\leq|(1-L)^{\frac{1}{2}}u|_2\cdot|(1-L)^{\frac{1}{2}}v\rangle|_2\\
&&\!\!\!\!\!\!\!\!=\|u\|_{F_{1,2}}\cdot\|v\|_{F_{1,2}}.
\end{eqnarray*}
This implies
$$\|(1-L)u\|_{F_{1,2}^*}\leq \|u\|_{F_{1,2}}.$$
Now we would like to identify $F_{1,2}^*$ with its dual $F_{1,2}$
via the corresponding Riesz isomorphism $R: F^*_{1,2}\rightarrow
F_{1,2}$ defined by $Rx=\langle x, ~\cdot\rangle_{F^*_{1,2}}$, $x\in
F^*_{1,2}$.

\begin{lemma}\label{lemma1}
The map $(1-L): F_{1,2}\rightarrow F^*_{1,2}$ is an isometric
isomorphism. In particular,
\begin{equation}\label{equation2}
\big\langle (1-L)u,~ (1-L)v \big\rangle_{F^*_{1,2}}=\langle u,~ v
\rangle_{F_{1,2}} \ \ \text{for all}\ \ u, v \in F_{1,2}.
\end{equation}
Furthermore, $(1-L)^{-1}: F^*_{1,2}\rightarrow F_{1,2}$ is the Riesz
isomorphism for $F^*_{1,2}$, i.e., for every $u \in F^*_{1,2}$,
\begin{equation}\label{equation3}
\langle u,~\cdot\rangle_{F^*_{1,2}}=\!\!_{F_{1,2}}\langle
(1-L)^{-1}u,~ \cdot \rangle_{F^*_{1,2}}.
\end{equation}
\end{lemma}

\begin{proof}
For all $u, v\in F_{1,2}$, by \eref{eqnarray1} we know
$$\!\!_{F^*_{1,2}}\langle(1-L)u,~ v\rangle_{F_{1,2}}=\langle(1-L)^{\frac{1}{2}}u,
~(1-L)^{\frac{1}{2}}v\rangle_2=\langle u,~ v\rangle_{F_{1,2}},$$
i.e., $(1-L): F_{1,2}\rightarrow F^*_{1,2}$ is the Riesz isomorphism
for $F_{1,2}$.

In particular, for all $u, v\in F_{1,2}$, since the Riesz
isomorphism is isometric,
\begin{eqnarray}\label{eqnarray2}
\langle(1-L)u,~ (1-L)v\rangle_{F^*_{1,2}}=\langle
u,~v\rangle_{F_{1,2}}.
\end{eqnarray}

Furthermore, for all $u, v\in F^*_{1,2},$
$$\langle u,~ v\rangle_{F^*_{1,2}}=\langle(1-L)^{-1}u,~ (1-L)^{-1}v\rangle_{F_{1,2}}=\!\!_{F_{1,2}}\langle (1-L)^{-1}u,~ v\rangle_{F_{1,2}^*}.$$\hspace{\fill}$\Box$
\end{proof}
In this sense, we identify $F_{1,2}^*$ with $F_{1,2}$ via the Riesz
map $(1-L)^{-1}: F_{1,2}^*\rightarrow F_{1,2}$, thus
$F_{1,2}^*\equiv F_{1,2}$. Note that $L^2(\mu)$ can be considered as
a subset of $F_{1,2}^*$, since for $u\in L^2(\mu)$, the map
$$v\longmapsto\langle u, v\rangle_2,\ \ v\in F_{1,2},$$
belongs to $F_{1,2}^*$. Here $\langle\cdot ,\cdot \rangle_2$ denotes
the usual inner product on $L^2(\mu)$. Obviously, in this sense
$L^2(\mu)\subset F_{1,2}^*$ continuously and densely. Consequently,
we get a Gelfand triple with $V:=L^2(\mu)$, $H:=F_{1,2}^*$,
$$V=L^2(\mu)\subset F_{1,2}^*\subset (L^2(\mu))^*,$$
which satisfies
\begin{eqnarray}\label{eqnarray4}
&&_{V^*}\langle u, v\rangle_V=\langle u,
v\rangle_H,\text{~for~all~}u\in F_{1,2}^*, v\in L^2(\mu).
\end{eqnarray}

\begin{lemma}\label{lemma3}
The map
$$1-L:F_{1,2}\rightarrow F_{1,2}^*$$
extends to a linear isometry
$$1-L:L^2(\mu)\rightarrow(L^2(\mu))^*,$$
and for all $u,v\in L^2(\mu)$,
\begin{eqnarray}\label{eqnarray3}
_{(L^2(\mu))^*}\langle(1-L)u, v\rangle_{L^2(\mu)}=\int_Eu\cdot
v~d\mu.
\end{eqnarray}
\end{lemma}

\begin{proof}
Let $u\in F_{1,2}$. Since $(1-L)u\in F_{1,2}^*$, from
\eref{equation3} and \eref{eqnarray4} we obtain that for all $v\in
L^2(\mu)$,
\begin{eqnarray}\label{eqnarray5}
_{(L^2(\mu))^*}\!\langle(1-L)u,v\rangle_{L^2(\mu)}=\langle(1-L)u,v\rangle_{F_{1,2}^*}=_{F_{1,2}}\langle
u, v\rangle_{F_{1,2}^*}=\langle u,v\rangle_2,
\end{eqnarray}
the last equality holds since $F_{1,2}\subset L^2(\mu)\subset
F_{1,2}^*$ densely and continuously. Therefore,
$$|(1-L)u|_{(L^2(\mu))^*}\leq |u|_2.$$ In this sense, $1-L$ extends
to a continuous linear map
$$1-L:L^2(\mu)\rightarrow (L^2(\mu))^*$$ such that \eref{eqnarray5}
holds for all $u\in L^2(\mu)$, i.e., \eref{eqnarray3} is proved.
\vspace{2mm}

So, applying it to $u\in L^2(\mu)$ and
$$v:=|u|_2^{-1}u\in L^2(\mu),$$
by \eref{eqnarray5} we obtain that
$$_{V^*}\langle(1-L)u, v \rangle_V=\langle u, v\rangle_2=\langle u, |u|_2^{-1}u\rangle_2=|u|_2,$$
and $|v|_2=1$, so $|(1-L)u|_{V^*}=|u|_V$ and the assertion is
completely proved. \hspace{\fill}$\Box$
\end{proof}

\vspace{2mm} Consider the quadratic form $(\mathcal {E},D(\mathcal
{E}))$ on $L^2(\mu)$ associated with $(L,D(L))$, i.e.
\begin{eqnarray*}\label{eqnarray24}
D(\mathcal {E}):=F_{1,2}
\end{eqnarray*}
and
\begin{eqnarray*}\label{eqnarray25}
\mathcal {E}(u,v):=\mu\big(\sqrt{-L}u\sqrt{-L}v\big);\ u,\ v\in
F_{1,2}.
\end{eqnarray*}
From \cite{MR}, we know $(L,D(L))$ is indeed the associated
Dirichlet operator on $L^2(\mu)$.

\vspace{2mm} Throughout the paper, let
$L^2([0,T]\times\Omega;L^2(\mu))$ denote the space of all
$L^2(\mu)$-valued square-integrable functions on
$[0,T]\times\Omega$, and $C([0,T];F_{1,2}^*)$ the space of all
continuous $F_{1,2}^*$-valued functions on $[0,T]$. For two Hilbert
spaces $H_1$ and $H_2$, the space of Hilbert-Schmidt operators from
$H_1$ to $H_2$ is denoted by $L_2(H_1, H_2)$. For simplicity, the
positive constants $c$, $C$, $C_i$, $i=1,2,3$ used in this paper may
change from line to line. We would like to refer \cite{BDR1} for
more background information and results on SPMEs.

\section{The Main Result}
\setcounter{equation}{0}
 \setcounter{definition}{0}

Consider \eref{eq:1} under the following conditions:

\medskip
\noindent \textbf{(H1)} $\Psi(\cdot)\!: \mathbb{R}\rightarrow\!
\mathbb{R}$ is a monotonically nondecreasing Lipschitz function with
$\Psi(0)=0$.

\vspace{1mm}
\medskip
\noindent \textbf{(H2)} $B\!: [0, T]\times L^2(\mu)\times
\Omega\rightarrow L_2(L^2(\mu), F^*_{1,2})$ is progressively
measurable, i.e. for any $t\in[0,T]$, this mapping restricted to
$[0,t]\times L^2(\mu)\times \Omega$ is measurable w.r.t.
$\mathscr{B}([0,t])\times\mathscr{B}(L^2(\mu))\times \mathscr{F}_t$,
where $\mathscr{B}(\cdot)$ is the Borel $\sigma$-field for a
topological space. For simplicity, below we will write $B(t,u)$
meaning the mapping $\omega\mapsto B(t,u,\omega)$, and $B(t,u)$
satisfies

\vspace{2mm}

\noindent \textbf{(i)} there exists $C_1\in[0, \infty)$ satisfying
$$\|B(\cdot, u)-B(\cdot, v)\|^2_{L_2(L^2(\mu), F^*_{1,2})}\leq C_1\|u-v\|^2_{ F^*_{1,2}}\ \ \text{for\ all}\ u, v\in L^2(\mu)\ \text{on}\ [0, T]\times \Omega;$$

\noindent \textbf{(ii)} there exists $C_2\in(0, \infty)$ satisfiying
$$\|B(\cdot, u)\|^2_{L_2(L^2(\mu), F^*_{1,2})}\leq C_2\|u\|^2_{ F^*_{1,2}}\ \ \text{for\ all}\ u\in L^2(\mu)\ \text{on}\ [0, T]\times \Omega.$$

\noindent \textbf{(iii)} there exists $C_3\in(0, \infty)$ satisfiying
$$\|B(\cdot, u)\|^2_{L_2(L^2(\mu), L^2(\mu))}\leq C_3(|u|^2_2+1)\ \ \text{for\ all}\ u\in L^2(\mu)\ \text{on}\ [0, T]\times \Omega.$$

\medskip
\noindent \textbf{Remark 3.1}~ In the following, we cite one example
from \cite[Section 2]{RW} of $B$, which satisfies
\textbf{(H2)}\textbf{(ii)}.

\medskip
\noindent \textbf{(M)} ~Let $N\in \mathbb{N}\cup {+\infty}$ and
$e_k\in L^2(\mu)\cap L^\infty(\mu)$, $1\leq k\leq N$, be an
orthonormal system in $L^2(\mu)$ such that for every $1\leq k\leq N$
there exists $\xi_k\in(0,\infty)$ such that for all $a\in(0,\infty)$
\begin{eqnarray*}
 &&\big|_{F^*_{1,2}}\langle x, e_k u\rangle_{F_{1,2}}\big|\leq
 \xi_k\|x\|_{H_a}\mathcal
 {E}_a(u,u)^{\frac{1}{2}},~~\text{for~all}~u\in D(\mathcal {E}),
\end{eqnarray*}
where $\mathcal {E}_a:=a\mathcal {E}+\langle\cdot, \cdot\rangle$ on
$D(\mathcal {E})$ and $\|\cdot\|_{H_a}$ denotes the corresponding
norm on the dual space of $D(\mathcal {E})$.

\vspace{2mm} \textbf{(M)} just means that each $e_k$ is a multiplier
on $H_a$ with norm independent of $a>0$. Choose $\mu_k\in
(0,\infty)$ such that
\begin{eqnarray*}
 &&\sum_{k=1}^\infty \xi_k^2\mu^2_k<\infty,
\end{eqnarray*}
and define for $x\in H$, $B(x)\in L_2(L^2(\mu); H)$ by
\begin{eqnarray*}
&&B(x)h:=\sum_{k=1}^\infty \mu_k\langle e_k, h\rangle x\cdot e_k,~
h\in L^2(\mu).
\end{eqnarray*}
Indeed, (extending $\{e_k|k\in \mathbb{N}\}$ to an orthonormal basis
of $L^2(\mu)$) by \textbf{(M)} we have for $x\in H$, $a\in
(0,\infty)$
\begin{eqnarray*}
\|B(x)\|^2_{L_2(L^2(\mu), H_a)}
&&\!\!\!\!\!\!\!\!=\sum\limits_{k=1}^\infty\|B(x)e_k\|^2_{H_a}\\
&&\!\!\!\!\!\!\!\!=\sum_{k=1}^\infty\mu_k^2\|xe_k\|^2_{H_a}\\
&&\!\!\!\!\!\!\!\!\leq\sum_{k=1}^\infty\mu_k^2\xi_k^2\|x\|^2_{H_a}
\end{eqnarray*}
and since $x\rightarrow B(x)$ is linear and $V\subset H$, condition
\textbf{(H2)}\textbf{(ii)} follows. For more examples, we refer to
\cite[Section 2]{RW}.
\begin{definition}
Let $x\in F_{1,2}^*$. A continuous
$(\mathscr{F}_t)_{t\geq0}$-adapted process $X:[0,T]\rightarrow
F_{1,2}^*$ is called strong solution to \eref{eq:1} if the following
conditions are satisfied:
\begin{equation}\label{equ:1}
X\in L^2([0,T]\times\Omega;L^2(\mu))\cap
L^2(\Omega;C([0,T];F_{1,2}^*)),
\end{equation}
\begin{equation}\label{equ:2}
\int_0^\bullet \Psi(X(s))ds\in C([0,T];F_{1,2}),\
\mathbb{P}\text{-a.s.},
\end{equation}
\begin{equation}\label{equ:3}
X(t)-L\int_0^t\Psi(X(s))ds=x+\int_0^tB(s,X(s))dW(s),\ \forall\ t
\in[0,T],\ \mathbb{P}\text{-a.s.}.
\end{equation}
\end{definition}

\begin{theorem} \label{theorem1} Suppose \textbf{(H1)} and \textbf{(H2)} are satisfied.
Then, for each $x\in L^2(\mu)$, there is a unique strong solution
$X$ to \eref{eq:1} and exists $C\in[0,\infty)$ satisfying
$$\mathbb{E}\Big[\sup_{t\in[0, T]}\big|X(t)\big|_2^2\Big]\leq 2|x|_2^2e^{CT}.$$
Assume further that
\begin{equation}\label{eq:2}
\Psi(r)r\geq c r^2,\ \forall\ r\in \mathbb{R},
\end{equation}
where $c\in(0, \infty)$. Then, there is a unique strong solution $X$
to \eref{eq:1} for all $x\in F_{1,2}^*$.
\end{theorem}

For the proof of the above theorem, we firstly consider the
approximating equations for \eref{eq:1}:
\begin{equation} \label{eq:3}
\left\{ \begin{aligned}
&dX^\nu(t)+(\nu-L)\Psi(X^\nu(t))dt=B(t,X^\nu(t))dW(t),\ \text{in}\ (0,T)\times E,\\
&X^\nu(0)=x \ \text{on} \ E,
\end{aligned} \right.
\end{equation}
where $\nu\in(0,1)$. And we have the following result for
\eref{eq:3}.

\begin{lemma}\label{lemma2}
Suppose \textbf{(H1)} and \textbf{(H2)} are satisfied. Then, for
each $x\in L^2(\mu)$, there is a unique
$(\mathscr{F}_t)_{t\geq0}$-adapted solution to \eref{eq:3}, denoted
by $X^\nu$, i.e., in particular it has the following properties,
\begin{equation}\label{equ:4}
X^\nu\in L^2\big([0,T]\times\Omega;L^2(\mu)\big)\cap
L^2\big(\Omega;C([0,T];F_{1,2}^*)\big),
\end{equation}
\begin{equation}\label{equ:5}
X^\nu(t)+(\nu-L)\int_0^t\Psi(X^\nu(s))ds=x+\int_0^tB(s,X^\nu(s))dW(s),\
\forall t\in[0,T],~~ \mathbb{P}-a.s..
\end{equation}
Furthermore, there exists $C\in(0,\infty)$ such that for all
$\nu\in(0,1)$,
\begin{equation}\label{equ:6}
\mathbb{E} \Big[\sup_{t\in[0, T]}|X^\nu(t)|_2^2\Big]\leq
2|x|_2^2e^{CT}.
\end{equation}
In addition, if \eref{eq:2} is satisfied, there is a unique solution
$X^\nu$ to \eref{eq:3} satisfying \eref{equ:4} and \eref{equ:5} for
all $x\in F_{1,2}^*$.
\end{lemma}

\begin{proof}We proceed in two steps. In Step 1, we consider
the case when the initial value $x\in F_{1,2}^*$ and that
\eref{eq:2} is satisfied. In Step 2, we will prove the existence and
uniqueness result when $x\in L^2(\mu)$ and without assumption
\eref{eq:2}, by replacing $\Psi$ with $\Psi+\lambda I$,
$\lambda\in(0,1)$ and then letting $\lambda\rightarrow 0$.

\vspace{1mm}
\medskip
\noindent \textbf{Step 1:}
 Assume $x\in F_{1,2}^*$ and that
\eref{eq:2} is satisfied. Set $V\!\!:=L^2(\mu)$, $H\!\!:=F_{1,2}^*$,
$A u\!\!:=(L-\nu)\Psi(u)$ for $u\in V$. The space $F_{1,2}^*$ is
equipped with the equivalent norm
$$\|\eta\|_{F_{1,2,\nu}^*}:=\langle\eta, (\nu-L)^{-1}\eta\rangle^{\frac{1}{2}},~~\eta\in F_{1,2}^*.$$
Under the Gelfand triple $V\subset H\subset V^*$, we shall prove the
existence and uniqueness of the solution to \eref{eq:3} by using
\cite[Theorem 4.2.4]{LR} (or \cite[Theorem 4.2.4]{PM}).

\vspace{2mm} In the following, we shall verify the four conditions
of the existence and uniqueness theorem in \cite{LR, PM}.

\medskip
\textbf{(i)} (Hemicontinuity)

Let $u, v, w\in V=L^2(\mu)$. We have to show for
$\lambda\in\mathbb{R}$, $|\lambda|\leq1$,
\begin{eqnarray*}
\lim_{\lambda\rightarrow0}\ _{V^*}\!\langle A(u+\lambda v),
w\rangle_V-_{V^*}\!\!\langle Au, w\rangle_V=0.
\end{eqnarray*}
By Lemma 2.2
\begin{eqnarray*}
&&_{V^*}\langle A(u+\lambda v), w\rangle_V\\
=\!\!\!\!\!\!\!\!&&_{V^*}\!\langle(L-\nu)\Psi(u+\lambda v),w\rangle_V\\
=\!\!\!\!\!\!\!\!&&-_{V^*}\!\langle(1-L)\Psi(u+\lambda v),w\rangle_V+(1-\nu)_{V^*}\!\big\langle(1-L)(1-L)^{-1}\Psi(u+\lambda v),w\big\rangle_V\\
=\!\!\!\!\!\!\!\!&&-\langle\Psi(u+\lambda
v),w\rangle_2+(1-\nu)\langle(1-L)^{-1}\Psi(u+\lambda v),w\rangle_2\\
=\!\!\!\!\!\!\!\!&&-\int_E \Psi(u+\lambda v)\cdot w
d\mu+(1-\nu)\int_E(1-L)^{-1}\Psi(u+\lambda v)\cdot wd\mu.
\end{eqnarray*}
By the Lipschitz continuity of $\Psi$ and denoting $k:=Lip\Psi$, the
first integrand in the right-hand side of the above equality is
bounded by
\begin{eqnarray*}
&&|\Psi(u+\lambda v)|\cdot|w|\leq k(|u|+|v|)\cdot|w|,
\end{eqnarray*}
which by H\"{o}lder's inequality is in $L^1(\mu)$. Since
$(1-L)^{-1}$ is a contraction on $L^2(\mu)$ (\cite[Chapter I]{MR}), in order to prove the convergence of
$(1-L)^{-1}\Psi(u+\lambda v)\cdot w$ in $L^1(\mu)$, it is sufficient
to show the convergence of $\Psi(u+\lambda v)$ in $L^2(\mu)$, which
is obvious because $\Psi$ is Lipschitz and
$$|\Psi(u+\lambda v)|\leq k(|u|+|v|).$$

\medskip
\textbf{(ii)} (Weak Monotonicity)

Let $u, v\in V=L^2(\mu)$, then by Lemma 2.2 and \eref{eqnarray4}
\begin{eqnarray}\label{eqna2}
&&2_{V^*}\langle Au-A v, u-v\rangle_V+\|B(\cdot,u)-B(\cdot, v)\|^2_{L_2(L^2(\mu),F^*_{1,2})}\nonumber\\
&&\!\!\!\!\!\!\!\!=2_{V^*}\big\langle(L-\nu)(\Psi(u)-\Psi(v)), u-v\big\rangle_V+\|B(\cdot,u)-B(\cdot, v)\|^2_{L_2(L^2(\mu),F^*_{1,2})}\nonumber\\
&&\!\!\!\!\!\!\!\!=-2_{V^*}\langle (1-L)(\Psi(u)-\Psi(v)),
u-v\rangle_V+2\ (1-\nu)_{V^*}\big\langle\Psi(u)-\Psi(v),
u-v\big\rangle_V\nonumber\\
&&\!\!\!\!\!\!\!\!~~~ +\|B(\cdot,u)-B(\cdot,
v)\|^2_{L_2(L^2(\mu),F^*_{1,2})}\nonumber\\
&&\!\!\!\!\!\!\!\!=-2\big\langle (\Psi(u)-\Psi(v)),
u-v\big\rangle_2+2(1-\nu)\big\langle\Psi(u)-\Psi(v),
u-v\big\rangle_{F^*_{1,2}}\nonumber\\
&&\!\!\!\!\!\!\!\!~~~ +\|B(\cdot,u)-B(\cdot,
v)\|^2_{L_2(L^2(\mu),F^*_{1,2})}.
\end{eqnarray}
Set $\tilde{\alpha}:=(Lip\Psi+1)^{-1}$. By assumption \noindent
\textbf{(H1)} on $\Psi$, we know that
\begin{eqnarray}\label{eqn2}
\big(\Psi(r)-\Psi(r')\big)(r-r')\geq
\tilde{\alpha}|\Psi(r)-\Psi(r')|^2, \ \ \forall r,\ r'\in\
\mathbb{R}.
\end{eqnarray}
Since $L^2(\mu)\subset F^*_{1,2}$ continuously, by Young's
inequality
\begin{eqnarray}\label{eqn3}
&&\big\langle\Psi(u)-\Psi(v),u-v\big\rangle_{F^*_{1,2}}\nonumber\\
&&\!\!\!\!\!\!\!\!\leq\|\Psi(u)-\Psi(v)\|_{F^*_{1,2}}\cdot\|u-v\|_{F^*_{1,2}}\nonumber\\
&&\!\!\!\!\!\!\!\!\leq |\Psi(u)-\Psi(v)|_2\cdot\|u-v\|_{F^*_{1,2}}\nonumber\\
&&\!\!\!\!\!\!\!\!\leq\frac{\tilde{\alpha}}{1-\nu}\big|\Psi(u)-\Psi(v)\big|_2^2+\frac{1-\nu}{\tilde{\alpha}}\|u-v\|^2_{F^*_{1,2}}.
\end{eqnarray}
By \noindent \textbf{(H2)} \noindent \textbf{(i)}, and taking
\eref{eqn2}, \eref{eqn3} into account, \eref{eqna2} is dominated by
\begin{eqnarray*}
&&-2\tilde{\alpha}\big|\Psi(u)-\Psi(v)\big|^2_2+2\tilde{\alpha}\big|\Psi(u)-\Psi(v)\big|^2_2+\frac{2(1-\nu)^2}{\tilde{\alpha}}\|u-v\|^2_{F^*_{1,2}}+C_1\|u-v\|^2_{F^*_{1,2}}\\
&&\!\!\!\!\!\!\!\!=
\Big[\frac{2(1-\nu)^2}{\tilde{\alpha}}+C_1\Big]\cdot\|u-v\|^2_{F^*_{1,2}}.
\end{eqnarray*}
Hence weak monotonicity holds.

\medskip
\textbf{(iii)} (Coercivity)

Let $u\in L^2(\mu)$. By Lemma 2.2 and \eref{eqnarray4}
\begin{eqnarray}\label{eqn4}
&&2_{V^*}\langle Au, u\rangle_V+\|B(\cdot, u)\|^2_{L_2(L^2(\mu), F^*_{1,2})}\nonumber\\
&&\!\!\!\!\!\!\!\!=-2_{V^*}\big\langle(1-L)\Psi(u),u\big\rangle_V+2(1-\nu)_{V^*}\langle\Psi(u),u\rangle_V+\|B(\cdot, u)\|^2_{L_2(L^2(\mu), F^*_{1,2})}\nonumber\\
&&\!\!\!\!\!\!\!\!=-2\langle\Psi(u),u\rangle_2+2(1-\nu)\langle\Psi(u),u\rangle_{F^*_{1,2}}+\|B(\cdot,
u)\|^2_{L_2(L^2(\mu), F^*_{1,2})}.
\end{eqnarray}
By \eref{eq:2}
\begin{eqnarray}\label{eqn5}
-2\langle\Psi(u),u\rangle_2=-2\int_E \Psi(u)\cdot u d\mu\leq
-2c|u|^2_2.
\end{eqnarray}
Since $L^2(\mu)\subset F^*_{1,2}$ continuously, by Young's
inequality for $\varepsilon\in(0,1)$
\begin{eqnarray}\label{eqn6}
\big\langle\Psi(u),u\big\rangle_{F^*_{1,2}}&&\!\!\!\!\!\!\!\!\leq\|\Psi(u)\|_{F^*_{1,2}}\cdot\|u\|_{F^*_{1,2}}\nonumber\\
&&\!\!\!\!\!\!\!\!\leq |\Psi(u)|_2\cdot\|u\|_{F^*_{1,2}}\nonumber\\
&&\!\!\!\!\!\!\!\!\leq
\varepsilon^2k^2|u|_2^2+\frac{1}{\varepsilon^2}\|u\|^2_{F^*_{1,2}}.
\end{eqnarray}
By \noindent \textbf{(H2)} \noindent \textbf{(ii)}, and taking
\eref{eqn5} and \eref{eqn6} into account, \eref{eqn4} is dominated
by
\begin{eqnarray}\label{eqn7}
\big[-2c+2\varepsilon^2k^2(1-\nu)\big]\cdot|u|_2^2+\Big[\frac{2(1-\nu)}{\varepsilon^2}+C_2\Big]\cdot\|u\|^2_{F^*_{1,2}}.\nonumber
\end{eqnarray}
Choosing $\varepsilon$ small enough, $-2c+2\varepsilon^2k^2(1-\nu)$
becomes negative, which implies the coercivity.

\medskip
\textbf{(iv)} (Boundedness)

Let $u\in L^2(\mu)$. Since
$$|Au|_{V^*}=|(L-\nu)\Psi(u)|_{V^*}=\sup_{|v|_2=1}\ _{V^*}\!\langle
(L-\nu)\Psi(u), v\rangle_V,$$ by Lemma 2.2 and since $(1-L)^{-1}$ is
a contraction on $L^2(\mu)$, we deduce
\begin{eqnarray*}
&&_{V^*}\!\langle (L-\nu)\Psi(u), v\rangle_V\\
&&\!\!\!\!\!\!\!\!=-_{V^*}\langle(1-L)\Psi(u),v\rangle_V+(1-\nu)_{V^*}\big\langle(1-L)(1-L)^{-1}\Psi(u),v\big\rangle_V\\
&&\!\!\!\!\!\!\!\!=-\langle\Psi(u),v\rangle_2+(1-\nu)\langle(1-L)^{-1}\Psi(u),v\rangle_2\\
&&\!\!\!\!\!\!\!\!\leq|\Psi(u)|_2\cdot|v|_2+(1-\nu)|\Psi(u)|_2\cdot|v|_2.
\end{eqnarray*}
So
\begin{eqnarray*}
&&|Au|_{V^*}\leq2|\Psi(u)|_2\leq2k|u|_2.
\end{eqnarray*}
Hence the boundedness holds. \vspace{2mm}

By \cite[Theorem 4.2.4]{LR}, there exists a unique solution to
\eref{eq:3}, denoted by $X^\nu$, which takes values in $F^*_{1,2}$
and satisfies \eref{equ:4} and \eref{equ:5}.

\vspace{2mm}
\medskip
\textbf{Step 2:} If $\Psi$ does not satisfy \eref{eq:2}, the above
(i), (ii) and (iv) still hold, but (iii) not in general. In this
case, we will approximate $\Psi$ by $\Psi+\lambda I$,
$\lambda\in(0,1)$.

Consider the approximating equation:
\begin{equation}\label{eq:4}
\left\{ \begin{aligned}
&X^\nu_\lambda(t)+(\nu-L)\big(\Psi(X^\nu_\lambda(t))+\lambda X^\nu_\lambda(t)\big)dt=B(t,X^\nu_\lambda(t))dW(t),\ \text{in}\ [0,T]\times E,\\
&X^\nu_\lambda(0)=x \in F^*_{1,2}\ \ \text{on} \ E.
\end{aligned} \right.
\end{equation}

By \cite[Theorem 4.2.4]{LR}, it is easy to prove that there is a
solution $X^\nu_\lambda$ to \eref{eq:4} which satisfies
$X^\nu_\lambda\in L^2\big([0,T]\times\Omega;L^2(\mu)\big)\cap
L^2\big(\Omega;C([0,T];F_{1,2}^*)\big)$,
$$X^\nu_\lambda(t)+(\nu-L)\int_0^t\Psi(X^\nu_\lambda(t))+\lambda X^\nu_\lambda(t)~ds=x+\int_0^tB(s,X^\nu_\lambda(s))dW(s),~~~\mathbb{P}-a.s.$$
and
\begin{eqnarray}\label{eqna1}
\mathbb{E}\Big[\sup_{t\in[0,T]}\|X^\nu_\lambda(t)\|^2_{F^*_{1,2}}\Big]<\infty.
\end{eqnarray}

In the following, we want to prove that the sequence
$\{X^\nu_\lambda\}$ converges to the solution of \eref{eq:3} as
$\lambda\to 0$. From now on, we assume that the initial value $x\in
L^2(\mu)$.

\begin{claim}\label{cla:1}
$$\mathbb{E}\Big[\sup_{s\in[0,T]}\big|X^\nu_\lambda(s)\big|_2^2\Big]+4\lambda \mathbb{E}\int_0^t\big\|X^\nu_\lambda(s)\big\|^2_{F_{1,2}}ds\leq C_T(|x|^2_2+1),$$
where $C_T$ is independent of $\nu, \lambda\in(0,1)$.
\end{claim}

\begin{proof}
Rewrite \eref{eq:4}, for $t\in[0,T]$,
\begin{equation}\label{eq:5}
X^\nu_\lambda(t)=x+\int_0^t(L-\nu)\big(\Psi(X^\nu_\lambda(s))+\lambda
X^\nu_\lambda(s)\big)ds+\int_0^tB(s,X^\nu_\lambda(s))dW(s).
\end{equation}
For $\alpha>\nu$, applying the operator
$(\alpha-L)^{-\frac{1}{2}}:F^*_{1,2}\rightarrow L^2(\mu)$ to both
sides of the above equation, we get
\begin{eqnarray*}
&&(\alpha-L)^{-\frac{1}{2}}X^\nu_\lambda(t)\\
=\!\!\!\!\!\!\!\!&&(\alpha-L)^{-\frac{1}{2}}x+\int_0^t(L-\nu)(\alpha-L)^{-\frac{1}{2}}
\big(\Psi(X^\nu_\lambda(s))+\lambda X^\nu_\lambda(s)\big)ds\\
&&+\int_0^t(\alpha-L)^{-\frac{1}{2}}B(s,X^\nu_\lambda(s))dW(s).
\end{eqnarray*}
Applying It\^{o}'s formula (\cite[Theorem 4.2.5]{LR}) with
$H=L^2(\mu)$, we obtain, for $t\in[0,T]$,
\begin{eqnarray}\label{eq:6}
&&\big|(\alpha-L)^{-\frac{1}{2}}X^\nu_\lambda(t)\big|^2_2\nonumber\\
=\!\!\!\!\!\!\!\!&&\big|(\alpha-L)^{-\frac{1}{2}}x\big|^2_2
+2\int_0^t\!\!\!\!
_{F^*_{1,2}}\big\langle(L-\nu)(\alpha-L)^{-\frac{1}{2}}\Psi(X^\nu_\lambda(s)),(\alpha-L)^{-\frac{1}{2}}
X^\nu_\lambda(s)\big\rangle_{F_{1,2}}ds\nonumber\\
&&+2\lambda\int_0^t\!\!\!\!
_{F^*_{1,2}}\big\langle(L-\nu)(\alpha-L)^{-\frac{1}{2}}X^\nu_\lambda(s),(\alpha-L)^{-\frac{1}{2}}
X^\nu_\lambda(s)\big\rangle_{F_{1,2}}\ds\nonumber\\
&&+\int_0^t\big\|(\alpha-L)^{-\frac{1}{2}}B(s,X^\nu_\lambda(s))\big\|^2_{L_2(L^2(\mu),L^2(\mu))}ds\nonumber\\
&&+2\int_0^t\big\langle(\alpha-L)^{-\frac{1}{2}}X^\nu_\lambda(s),(\alpha-L)^{-\frac{1}{2}}
B(s,X^\nu_\lambda(s))dW(s)\big\rangle_2.
\end{eqnarray}

To estimate the second term in the right hand side of \eref{eq:6}, set $P:=(\alpha-\nu)(\alpha-L)^{-1}$. For $f\in L^2(\mu)$, we have
\begin{eqnarray*}
(P-I)f=\!\!\!\!\!\!\!\!&&\big[(\alpha-L)^{-\frac{1}{2}}(\alpha-\nu)(\alpha-L)^{-\frac{1}{2}}
-(\alpha-L)^{-\frac{1}{2}}(\alpha-L)(\alpha-L)^{-\frac{1}{2}}\big]f\\
=\!\!\!\!\!\!\!\!&&\big[(\alpha-L)^{-\frac{1}{2}}(L-\nu)(\alpha-L)^{-\frac{1}{2}}\big]f.
\end{eqnarray*}
Since $L$ is the infinitesimal generator of a symmetric sub-Markovian strongly continuous contraction semigroup $(P_t)_{t\geq0}$ on $L^2(\mu)$, then, $P$ is a symmetric sub-Markovian contraction on $L^2(\mu)$. From \cite[Lemma 5.1 (i)]{RW}, there exists a probability kernel $p$ on $(E,\mathscr{B},\mu)$ such that for all $f\in L^2(\mu)$
\begin{eqnarray*}
Pf(\xi):=\int_Ef(\tilde{\xi})p(\xi,d\tilde{\xi}),~~\xi\in E.
\end{eqnarray*}
Applying \cite[Lemma 5.1 (ii)]{RW} (here the assumption that $(E,\mathscr{B},\mu)$ is a standard measurable space is needed) with $f:=X^\nu_\lambda(s)$ and
$g:=\Psi(X^\nu_\lambda(s))$, since $\Psi$ is monotone, $\Psi(0)=0$ and $P1\leq1$, one obtains
\begin{eqnarray}\label{eq2}
&&2\int_0^t\!\!\!\!
_{F^*_{1,2}}\big\langle(L-\nu)(\alpha-L)^{-\frac{1}{2}}\Psi(X^\nu_\lambda(s)),(\alpha-L)^{-\frac{1}{2}}
X^\nu_\lambda(s)\big\rangle\!
_{F_{1,2}}ds\nonumber\\
=\!\!\!\!\!\!\!\!&&2\int_0^t\big\langle \Psi(X^\nu_\lambda(s)),(P-I)X^\nu_\lambda(s)\big\rangle_2ds\nonumber\\
=\!\!\!\!\!\!\!\!&&-\int_0^t\int_E\int_E\big[\Psi(f(\widetilde{\xi}))-\Psi(f(\xi))\big]\big[f(\widetilde{\xi})
-f(\xi)\big]
p(\xi,d\widetilde{\xi})d\xi ds\nonumber\\
\!\!\!\!\!\!\!\!&&-2\int_0^t\int_E(1-P1)f(\xi)\cdot \Psi(f(\xi))d\xi ds\nonumber\\
\leq\!\!\!\!\!\!\!\!&&0.
\end{eqnarray}
For the second integral on the right hand side of \eref{eq:6}, rewrite $L-\nu=-(1-L)+(1-\nu)$, by \eref{eqnarray3} we have for all $\nu, \lambda\in(0,1)$
\begin{eqnarray}\label{eq1}
&&2\lambda\int_0^t\!\!\!\!
_{F^*_{1,2}}\big\langle(L-\nu)(\alpha-L)^{-\frac{1}{2}}X^\nu_\lambda(s),(\alpha-L)^{-\frac{1}{2}}X^\nu_\lambda(s)\big\rangle_{F_{1,2}}ds\nonumber\\
=\!\!\!\!\!\!\!\!&&-2\lambda\int_0^t\ _{F^*_{1,2}}\big\langle(1-L)(\alpha-L)^{-\frac{1}{2}}X^\nu_\lambda(s),(\alpha-L)^{-\frac{1}{2}}X^\nu_\lambda(s)\big\rangle_{F_{1,2}}ds\nonumber\\
&&+2\lambda(1-\nu)\int_0^t\ _{F^*_{1,2}}\big\langle(\alpha-L)^{-\frac{1}{2}}X^\nu_\lambda(s),(\alpha-L)^{-\frac{1}{2}}X^\nu_\lambda(s)\big\rangle_{F_{1,2}}ds\nonumber\\
=\!\!\!\!\!\!\!\!&&-2\lambda\int_0^t\big\|(\alpha-L)^{-\frac{1}{2}}X^\nu_\lambda(s)\big\|^2_{F_{1,2}}ds+2\lambda(1-\nu)\int_0^t|(\alpha-L)^{-\frac{1}{2}}X^\nu_\lambda(s)|^2_2ds\nonumber\\
\leq\!\!\!\!\!\!\!\!&&-2\lambda\int_0^t\big\|(\alpha-L)^{-\frac{1}{2}}X^\nu_\lambda(s)\big\|^2_{F_{1,2}}ds+2\int_0^t|(\alpha-L)^{-\frac{1}{2}}X^\nu_\lambda(s)|^2_2ds.
\end{eqnarray}
Multiplying both sides of \eref{eq:6} by $\alpha$, taking \eref{eq2} and \eref{eq1} into account, we have for all $t\in[0, T]$,
\begin{eqnarray}\label{eq3}
&&\big|\sqrt{\alpha}(\alpha-L)^{-\frac{1}{2}}X^\nu_\lambda(t)\big|_2^2+2\lambda\int_0^t\big\|\sqrt{\alpha}(\alpha-L)^{-\frac{1}{2}}X^\nu_\lambda(s)\big\|^2_{F_{1,2}}ds\nonumber\\
\leq\!\!\!\!\!\!\!\!&&\big|\sqrt{\alpha}(\alpha-L)^{-\frac{1}{2}}x\big|_2^2+2\int_0^t|\sqrt{\alpha}(\alpha-L)^{-\frac{1}{2}}X^\nu_\lambda(s)|_2^2ds\nonumber\\
\!\!\!\!\!\!\!\!&&+\int_0^t\big\|\sqrt{\alpha}(\alpha-L)^{-\frac{1}{2}}B(s,X^\nu_\lambda(s))\big\|^2_{L_2(L^2(\mu),L^2(\mu))}ds\nonumber\\
\!\!\!\!\!\!\!\!&&+2\int_0^t\big\langle\sqrt{\alpha}(\alpha-L)^{-\frac{1}{2}}X^\nu_\lambda(s),
\sqrt{\alpha}(\alpha-L)^{-\frac{1}{2}}B(s,X^\nu_\lambda(s))dW(s)\big\rangle_2.
\end{eqnarray}

Before doing further estimates, we need to prove that
\begin{eqnarray}\label{contraction}
\sqrt{\alpha}(\alpha-L)^{-\frac{1}{2}} \text{~ is~ a~ contraction~ on}~ L^2(\mu),
\end{eqnarray}
and
\begin{eqnarray}\label{convergence}
|\sqrt{\alpha}(\alpha-L)^{-\frac{1}{2}}u|^2_2\rightarrow |u|_2~~~\text{in}~~L^2(\mu)~~~\text{as}~~\alpha\rightarrow\infty.
\end{eqnarray}
These are true because $\alpha(\alpha-L)^{-1}$ is a contraction on $L^2(\mu)$, then from \cite[Page:8, Proposition 1.3 (ii)]{MR}, we know that $|\alpha(\alpha-L)^{-1}|\leq1$ and for all $u\in L^2(\mu)$,
\begin{eqnarray}\label{contraction1}
\alpha(\alpha-L)^{-1}u\rightarrow u~~~\text{in}~~L^2(\mu)~~~\text{as}~~\alpha\rightarrow\infty,
\end{eqnarray}
so,
\begin{eqnarray}\label{contraction2}
|\sqrt{\alpha}(\alpha-L)^{-\frac{1}{2}}u|^2_2&&\!\!\!\!\!\!\!\!=\langle \sqrt{\alpha}(\alpha-L)^{-\frac{1}{2}}u,\sqrt{\alpha}(\alpha-L)^{-\frac{1}{2}}u\rangle_2\nonumber\\
=\!\!\!\!\!\!\!\!&&_{F_{1,2}}\langle \alpha(\alpha-L)^{-1}u,u\rangle_{F^*_{1,2}}\nonumber\\
=\!\!\!\!\!\!\!\!&&\langle \alpha(\alpha-L)^{-1}u,u\rangle_2\nonumber\\
\leq \!\!\!\!\!\!\!\!&&|\alpha(\alpha-L)^{-1}u|_2\cdot|u|_2\nonumber\\
\leq\!\!\!\!\!\!\!\!&& |u|_2^2,
\end{eqnarray}
which indicates \eref{contraction}, from the third step of \eref{contraction2} and \eref{contraction1} we can get \eref{convergence}.

By \eref{contraction2} and \textbf{(H2)}\textbf{(iii)}, we have
\begin{eqnarray}\label{B}
&&\int_0^t\big\|\sqrt{\alpha}(\alpha-L)^{-\frac{1}{2}}B(s, X^\nu_\lambda(s))\big\|^2_{L_2(L^2(\mu),L^2(\mu))}ds\nonumber\\
\le\!\!\!\!\!\!\!\!&&\int_0^t\big\|B(s, X^\nu_\lambda(s))\big\|^2_{L_2(L^2(\mu),L^2(\mu))}ds\nonumber\\
\le\!\!\!\!\!\!\!\!&& C_3\int_0^t(|X^\nu_\lambda(s)|^2_2+1)ds.
\end{eqnarray}
Taking \eref{B} into \eref{eq3}, by \eref{contraction2}, and using the Burkholder-Davis-Gundy (BDG) inequality (with $p=1$), we
obtain
\begin{eqnarray}\label{eqnarray11}
&&\mathbb{E}\Big[\sup_{s\in[0,t]}\big|\sqrt{\alpha}(\alpha-L)^{-\frac{1}{2}}X^\nu_\lambda(s)\big|_2^2\Big]+2\lambda\mathbb{E}\int_0^t\big\|\sqrt{\alpha}(\alpha-L)^{-\frac{1}{2}}X^\nu_\lambda(s)\big\|^2_{F_{1,2}}ds\nonumber\\
\le\!\!\!\!\!\!\!\!&&\big|\sqrt{\alpha}(\alpha-L)^{-\frac{1}{2}}x\big|_2^2+2\mathbb{E}\int_0^t|X^\nu_\lambda(s)|^2_2ds+C_3\int_0^t(|X^\nu_\lambda(s)|^2_2+1)ds\nonumber\\
&&+6\mathbb{E}\left[\int_0^t\big|\sqrt{\alpha}(\alpha-L)^{-\frac{1}{2}}X^\nu_\lambda(s)\big|_2^2\cdot\big|\sqrt{\alpha}
(\alpha-L)^{-\frac{1}{2}}B(s,X^\nu_\lambda(s))\big|_{L_2(L^2(\mu),L^2(\mu))}^2ds\right]^{\frac12}.
\end{eqnarray}
The last term of the right hand side of the above inequality can be
estimated by
\begin{eqnarray}\label{eqnarray12}
&&6\mathbb{E}\left[\sup_{s\in[0,t]}\big|\sqrt{\alpha}(\alpha-L)^{-\frac{1}{2}}X^\nu_\lambda(s)\big|_2^2\cdot\int_0^t\big|\sqrt{\alpha}(\alpha-L)^{-\frac{1}{2}}B(s,X^\nu_\lambda(s))\big|_{L^2(\mu),L^2(\mu))}^2ds\right]^{\frac{1}{2}}\nonumber\\
\le\!\!\!\!\!\!\!\!&&\frac{1}{2}\mathbb{E}\Big[\sup_{s\in[0,t]}\big|\sqrt{\alpha}(\alpha-L)^{-\frac{1}{2}}X^\nu_\lambda(s)\big|_2^2\Big]+C_3\mathbb{E}\int_0^t(|X^\nu_\lambda(s)|^2_2+1)ds.
\end{eqnarray}
Taking \eref{eqnarray12} into \eref{eqnarray11}, we obtain that for $t\in[0,T]$,
\begin{eqnarray}\label{eqn1}
&&\mathbb{E}\left[\sup_{s\in[0,t]}\big|\sqrt{\alpha}(\alpha-L)^{-\frac{1}{2}}X^\nu_\lambda(s)\big|_2^2\right]+2\lambda\mathbb{E}\int_0^t\big\|\sqrt{\alpha}(\alpha-L)^{-\frac{1}{2}}X^\nu_\lambda(s)\big\|^2_{F_{1,2}}ds\nonumber\\
\le\!\!\!\!\!\!\!\!&&\big|\sqrt{\alpha}(\alpha-L)^{-\frac{1}{2}}x\big|_2^2+C\mathbb{E}\int_0^t\big(|X^\nu_\lambda(s)|^2_2+1)ds
\nonumber\\
&&+\frac{1}{2}\mathbb{E}\left[\sup_{s\in[0,t]}\big|\sqrt{\alpha}(\alpha-L)^{-\frac{1}{2}}X^\nu_\lambda(t)\big|_2^2\right].
\end{eqnarray}
Note that the first summand of the left hand side of the above
inequality is finite by \eref{eqna1}, since
$|\sqrt{\alpha}(\alpha-L)^{-\frac{1}{2}}\cdot|_2$ is equivalent to
$\|\cdot\|_{F^*_{1,2}}$. \eref{eqn1} shows that
\begin{eqnarray}\label{eqnarray13}
&&\mathbb{E}\left[\sup_{s\in[0,t]}\big|\sqrt{\alpha}(\alpha-L)^{-\frac{1}{2}}X^\nu_\lambda(s)\big|_2^2\right]
+4\lambda\mathbb{E}\int_0^t\big\|\sqrt{\alpha}(\alpha-L)^{-\frac{1}{2}}X^\nu_\lambda(s)\big\|^2_{F_{1,2}}ds\nonumber\\
\le\!\!\!\!\!\!\!\!&&2\big|\sqrt{\alpha}(\alpha-L)^{-\frac{1}{2}}x\big|_2^2+C\mathbb{E}\int_0^t(|X^\nu_\lambda(s)|_2^2+1)ds.
\end{eqnarray}

Now we want to take $\alpha\rightarrow\infty$, by \eref{convergence} we can easily estimate the first term in both sides of \eref{eqnarray13}, to estimate the second term in the left hand-side, we introduce
\begin{equation*}
\mathcal{E}^{ex}(u,u):=\left\{
                         \begin{array}{ll}
                           \mathcal{E}(u,u),~~~ \text{if}~~~ u\in F_{1,2}; \\
                           +\infty, ~~~~~~~\text{if}~~~ u\in L^2(\mu)\setminus F_{1,2},
                         \end{array}
                       \right.
\end{equation*}
and prove that $L^2(\mu)\ni u\mapsto \mathcal{E}^{ex}(u,u)$ is lower semicontinuous on $L^2(\mu)$. To clarify this, recall from \cite{Fukushima} that
\begin{eqnarray*}
\|u\|_{F_{1,2}}^2:=\mathcal{E}_1(u,u)=\mathcal{E}(u,u)+|u|_2^2,~~~ \forall u\in F_{1,2}~(=D(\mathcal{E})).
\end{eqnarray*}
Without lose of generality, let us consider a sequence $\{u_n\}_{n\in\Bbb{N}}\subset F_{1,2}$ such that
$$\liminf_{n\rightarrow\infty}\mathcal{E}(u_{n},u_{n})<\infty,$$ and with an element $u\in L^2(\mu)$ such that $u_n\rightarrow u$ in $L^2(\mu)$ as $n\rightarrow\infty$. Then there exists a subsequence $\{u_{n_k}\}_{k\in\Bbb{N}}$ such that
\begin{eqnarray*}
\lim_{k\rightarrow\infty}\mathcal{E}(u_{n_k},u_{n_k})=\liminf_{n\rightarrow\infty}\mathcal{E}(u_{n},u_{n}):=C
\end{eqnarray*}
and
\begin{eqnarray*}
\mathcal{E}_1(u_{n_k},u_{n_k})\rightarrow C+|u|_2^2,~~\text{as}~~k\rightarrow\infty.
\end{eqnarray*}
Hence $\{\|u_{n_k}\|^2_{F_{1,2}}\}_{k\in \Bbb{N}}$ is bounded, again there exists a subsequence $\{u_{n_{k_l}}\}_{l\in\Bbb{N}}$ and $u_0\in F_{1,2}$ such that $u_{n_{k_l}}\rightharpoonup u_0$ in $(F_{1,2},\|\cdot\|_{F_{1,2}})$ as $l\rightarrow\infty$, from \cite[Page:184, Theorem 2.2]{MR} we know that there exists a subsequence (for simplicity here we use the same notation) such that the Cesaro mean
\begin{eqnarray*}
\frac{1}{N}\sum_{l=1}^Nu_{n_{k_l}}\rightarrow u_0~~\text{strongly~in}~~(F_{1,2},\|\cdot\|_{F_{1,2}}),~~\text{as}~~N\rightarrow\infty,
\end{eqnarray*}
hence also in $L^2(\mu)$. So $u=u_0$, $u\in F_{1,2}$, $u_{n_{k_l}}\rightharpoonup u$ in $(F_{1,2},\|\cdot\|_{F_{1,2}})$ as $l\rightarrow\infty$, and because $(F_{1,2},\|\cdot\|_{F_{1,2}})$ is a Hilbert space, thus by the weakly lower semicontinuity of norms we have
\begin{eqnarray*}
\mathcal{E}(u,u)+|u|_2^2=\|u\|^2_{F_{1,2}}&&\!\!\!\!\!\!\!\!\leq \liminf_{l\rightarrow\infty}\|u_{n_{k_l}}\|^2_{F_{1,2}}\nonumber\\
&&\!\!\!\!\!\!\!\!=\liminf_{l\rightarrow\infty}\Big(\mathcal{E}(u_{n_{k_l}},u_{n_{k_l}})+|u_{n_{k_l}}|_2^2\Big)\nonumber\\
&&\!\!\!\!\!\!\!\!=\liminf_{l\rightarrow\infty}\mathcal{E}(u_{n_{k_l}},u_{n_{k_l}})+|u|_2^2.
\end{eqnarray*}
Therefore, we have
\begin{eqnarray*}
\mathcal{E}^{ex}(u,u)\leq\liminf_{n\rightarrow\infty}\mathcal{E}(u_n,u_n),~\forall u\in L^2(\mu),
\end{eqnarray*}
and furthermore,
\begin{eqnarray}\label{2}
\mathcal{E}^{ex}_1(u,u):=\mathcal{E}^{ex}(u,u)+|u|_2^2\leq\liminf_{n\rightarrow\infty}\mathcal{E}_1(u_n,u_n)=\liminf_{n\rightarrow\infty}\|u_n\|^2_{F_{1,2}},~\forall u\in L^2(\mu).
\end{eqnarray}

Let us continue to estimate \eref{eqnarray13}, letting $\alpha\to\infty$, by \eref{convergence}, \eref{2} and Fatou's lemma,
\begin{eqnarray}\label{claim111}
&&\!\!\!\!\!\!\!\!\mathbb{E}\Big[\sup_{s\in[0,T]}\big|X^\nu_\lambda(s)\big|_2^2\Big]+4\lambda\mathbb{E}\int_0^t\big\|X^\nu_\lambda(s)\big\|^2_{F_{1,2}}ds\nonumber\\
\leq&&\!\!\!\!\!\!\!\!\Bbb{E}\left[\sup_{s\in[0,t]}\liminf_{\alpha\rightarrow\infty}\big|\sqrt{\alpha}(\alpha-L)^{-\frac{1}{2}}X^\nu_\lambda(s)\big|_2^2\right]
+4\lambda\mathbb{E}\int_0^t\liminf_{\alpha\rightarrow\infty}\big\|\sqrt{\alpha}(\alpha-L)^{-\frac{1}{2}}X^\nu_\lambda(s)\big\|^2_{F_{1,2}}ds\nonumber\\
\leq&&\!\!\!\!\!\!\!\!\liminf_{\alpha\rightarrow\infty}\left\{\Bbb{E}\Big[\sup_{s\in[0,t]}\big|\sqrt{\alpha}(\alpha-L)^{-\frac{1}{2}}X^\nu_\lambda(s)\big|_2^2\Big]
+4\lambda\mathbb{E}\int_0^t\big\|\sqrt{\alpha}(\alpha-L)^{-\frac{1}{2}}X^\nu_\lambda(s)\big\|^2_{F_{1,2}}ds\right\}\nonumber\\
\leq&&\!\!\!\!\!\!\!\!\liminf_{\alpha\rightarrow\infty}\left[2\big|\sqrt{\alpha}(\alpha-L)^{-\frac{1}{2}}x\big|_2^2+C\mathbb{E}\int_0^t(|X^\nu_\lambda(s)|_2^2+1)ds\right]\nonumber\\
=&&\!\!\!\!\!\!\!\!2|x|_2^2+C\mathbb{E}\int_0^t\big(|X^\nu_\lambda(s)|_2^2+1\big)ds.
\end{eqnarray}
Then Gronwall's inequality yields the result.
\hspace{\fill}$\Box$
\end{proof}

\begin{claim}\label{cla:2}
$\{X^\nu_\lambda\}_{\lambda\in(0,1)}$ converges to an element
$X^\nu\in L^2([0,T]\times\Omega;L^2(\mu))$ as $\lambda\to0$.
\end{claim}

\begin{proof}
By It\^{o}'s formula we get that, for $\lambda,\lambda'\in(0,1)$ and
$t\in[0,T]$,
\begin{eqnarray}\label{eq:7}
&&\|X^\nu_\lambda(t)-X^\nu_{\lambda'}(t)\|^2_{F^*_{1,2,\nu}}\nonumber\\
&&+2\int_0^t\big\langle\Psi(X^\nu_\lambda(s)-\Psi(X^\nu_{\lambda'}(s)+\lambda
X^\nu_\lambda(s)-\lambda'X^\nu_{\lambda'}(s),X^\nu_\lambda(s)-X^\nu_{\lambda'}(s))\big\rangle_2ds\nonumber\\
=\!\!\!\!\!\!\!\!&&\int_0^t\big\|B(s,X^\nu_\lambda(s))-B(s,X^\nu_{\lambda'}(s))\big\|^2_{L_2(L^2(\mu),F^*_{1,2,\nu})}ds\nonumber\\
&&+2\int_0^t\big\langle
X^\nu_\lambda(s)-X^\nu_{\lambda'}(s),\big(B(s,X^\nu_\lambda(s))-B(s,X^\nu_{\lambda'}(s))\big)dW(s)\big\rangle_{F^*_{1,2,\nu}}.
\end{eqnarray}
\eref{eqn2} implies that for the second term on the left hand side
in \eref{eq:7} we have
\begin{eqnarray}\label{eqna3}
&&2\int_0^t\big\langle\Psi(X^\nu_\lambda(s))-\Psi(X^\nu_{\lambda'}(s))+\lambda
X^\nu_\lambda(s)-\lambda'X^\nu_{\lambda'}(s),X^\nu_\lambda(s)-X^\nu_{\lambda'}(s)\big\rangle_2ds\nonumber\\
\geq\!\!\!\!\!\!\!\!&&2\tilde{\alpha}\int_0^t\big|\Psi(X^\nu_\lambda(s))-\Psi(X^\nu_{\lambda'}(s))\big|_2^2ds\nonumber\\
~~\!\!\!\!\!\!\!\!&&+2\int_0^t\big\langle\lambda
X^\nu_\lambda(s)-\lambda'X^\nu_{\lambda'}(s),X^\nu_\lambda(s)-X^\nu_{\lambda'}(s)\big\rangle_2.
\end{eqnarray}
The assumption \textbf{(H2)}\textbf{(i)} yields
\begin{eqnarray}\label{eqnal1}
\int_0^t\big\|B(s,X^\nu_\lambda(s))-B(s,X^\nu_{\lambda'}(s))\big\|^2_{L_2(L^2(\mu),F^*_{1,2,\nu})}ds
\leq
C_1\int_0^t\big\|X^\nu_\lambda(s)-X^\nu_{\lambda'}(s)\big\|^2_{F^*_{1,2,\nu}}.
\end{eqnarray}
Using the BDG inequality and Young's inequality, for $t\in[0, T]$,
\eref{eq:7}-\eref{eqnal1} imply
\begin{eqnarray}
&&\mathbb{E}\Big[\sup_{s\in[0,t]}\big\|X^\nu_\lambda(s)-X^\nu_{\lambda'}(s)\big\|^2_{F^*_{1,2,\nu}}\Big]+2\tilde{\alpha}
\mathbb{E}\int_0^t\big|\Psi(X^\nu_\lambda(s))-\Psi(X^\nu_{\lambda'}(s))\big|_2^2ds\nonumber\\
\leq\!\!\!\!\!\!\!\!&& C_1\mathbb{E}\int_0^t\big\|X^\nu_\lambda(s)-X^\nu_{\lambda'}(s)\big\|^2_{F^*_{1,2,\nu}}ds\nonumber\\
&&-2\mathbb{E}\int_0^t\big\langle\lambda
X^\nu_\lambda(s)-\lambda'X^\nu_{\lambda'}(s),
X^\nu_\lambda(s)-X^\nu_{\lambda'}(s)\big\rangle_2ds\nonumber\\
&&+2\mathbb{E}\Big[\int_0^t\big\|X^\nu_\lambda(s)-X^\nu_{\lambda'}(s)\big\|^2_{F^*_{1,2,\nu}}\cdot
\big\|B(s, X^\nu_\lambda(s))- B(s,X^\nu_{\lambda'}(s))\big\|^2_{L_2(L^2(\mu),F^*_{1,2,\nu})} ds\Big]^{\frac{1}{2}}\nonumber\\
\leq\!\!\!\!\!\!\!\!&&\frac{1}{2}\mathbb{E}\Big[\sup_{s\in[0,t]}\big\|X^\nu_\lambda(s)-X^\nu_{\lambda'}(s)\big\|^2_{F^*_{1,2,\nu}}\Big]
+C\mathbb{E}\int_0^t\big\|X^\nu_\lambda(s)-X^\nu_{\lambda'}(s)\big\|^2_{F^*_{1,2,\nu}}ds\nonumber\\
&&+4(\lambda+\lambda')\mathbb{E}\int_0^t\big(\big|
X^\nu_\lambda(s)\big|_2^2+\big|X^\nu_{\lambda'}(s)\big|^2_2\big)ds.
\end{eqnarray}
Since $x\in L^2(\mu)$, Gronwall's lemma and Claim 3.1 imply for some
constant $C\!\!\in\!\!(0,\infty)$ independent of $\lambda, \lambda'$ such that
\begin{eqnarray}\label{eq:8}
\mathbb{E}\Big[\sup_{s\in[0,T]}\big\|X^\nu_\lambda(s)-X^\nu_{\lambda'}(s)\big\|^2_{F^*_{1,2}}\Big]+\mathbb{E}\int_0^T\big|\Psi(X^\nu_\lambda(s))-\Psi(X^\nu_{\lambda'}(s))\big|_2^2ds
\leq C(\lambda+\lambda').
\end{eqnarray}
\eref{eq:8} implies that there exists an $\mathscr{F}_t$-adapted
continuous $F^*_{1,2}$-valued process $\{X^\nu(t)\}_{t\in[0,T]}$
such that $X^\nu\in L^2\big(\Omega;C([0,T],F^*_{1,2})\big)$. This
together with Claim 3.1 implies that $X^\nu\in
L^2\big([0,T]\times\Omega;L^2(\mu)\big).$ \hspace{\fill}$\Box$
\end{proof}

\begin{claim}
$X^\nu$ satisfies \eref{equ:5}.
\end{claim}

\begin{proof}
From Claim 3.2, we know that
\begin{equation}\label{eq:10}
X_\lambda^\nu\to X^\nu\ \ \text{and}\ \ \int_0^\bullet
B(s,X_\lambda^\nu(s)) dW(s)\to\int_0^\bullet B(s,
X^\nu(s))dW(s),~~~~\lambda\to0
\end{equation}
in $L^2(\Omega;C([0,T],F^*_{1,2}))$. \eref{eq:5}, \eref{eq:10} yield
that
$$\int_0^\bullet\big(\Psi(X_\lambda^\nu(s))+\lambda
X_\lambda^\nu(s)\big)ds,\ \lambda>0,$$ converges to some element in
$L^2\big(\Omega;C([0,T],F_{1,2})\big)$ as $\lambda\to0$.  In
addition, by Claim 3.1, we have that, as $\lambda\to0$,
$$\int_0^\bullet \big(\Psi(X_\lambda^\nu(s))+\lambda X_\lambda^\nu(s)\big)ds\to\int_0^\bullet \Psi(X^\nu(s))ds$$
in $L^2\big(\Omega;L^2([0,T];L^2(\mu))\big)$. This and \eref{eq:10}
imply the claim. \hspace{\fill}$\Box$
\end{proof}

By lower semi-continuity of norms, \eref{equ:6} follows immediately from
Claim 3.1.

\vspace{2mm} {\bf Uniqueness} \vspace{2mm}

If $X^\nu_1$, $X^\nu_2$ are two solutions to \eref{equ:5}, we have
$\mathbb{P}-a.s.$
\begin{eqnarray}\label{equ:7}
&&\!\!\!\!\!\!\!\!X^\nu_1(t)-X^\nu_2(t)+(\nu-L)\int_0^t\Psi(X^\nu_1(s))-\Psi(X^\nu_2(s))ds\nonumber\\
=&&\!\!\!\!\!\!\!\!\int_0^t\big(B(s,X^\nu_1(s))-B(s,X^\nu_2(s))\big)dW(s),\
\forall\ t\in [0,T]
\end{eqnarray}
in $\Omega\times[0,T]\times E$. Apply It\^{o}'s formula in
$F^*_{1,2}$ to
$\big\|X^\nu_1(t)-X^\nu_2(t)\big\|^2_{F^*_{1,2,\nu}}$, we get
\begin{eqnarray}\label{eqnarray17}
&&\big\|X^\nu_1(t)-X^\nu_2(t)\big\|^2_{F^*_{1,2,\nu}}+2\int_0^t\big\langle\Psi(X^\nu_1(s))-\Psi(X^\nu_2(s)),X^\nu_1(s)-X^\nu_2(s)\big\rangle_2ds\nonumber\\
=\!\!\!\!\!\!\!\!&&\int_0^t\big\|B(s,X^\nu_1(s))-B(s,X^\nu_2(s))\big\|^2_{L_2(L^2(\mu),F^*_{1,2,\nu})}ds\nonumber\\
&&+2\int_0^t\big\langle X^\nu_1(s)-X^\nu_2(s),
\big(B(s,X^\nu_1(s))-B(s,X^\nu_2(s))\big)dW(s)\big\rangle_{F^*_{1,2,\nu}}.
\end{eqnarray}
The Lipschitz assumption \textbf{(H1)} on $\Psi$ implies that
\begin{eqnarray}\label{eqnarray18}
&&\big(\Psi(r)-\Psi(r')\big)(r-r')\geq (Lip
\Psi+1)^{-1}|\Psi(r)-\Psi(r')|^2,\ \text{for}\ r, r'\in \mathbb{R}.
\end{eqnarray}
Taking expectation of both sides to  \eref{eqnarray17}, then taking
\eref{eqnarray18} and \textbf{(H2)}\textbf{(i)} into account, we
obtain
\begin{eqnarray}\label{eqnarray19}
&&\!\!\!\!\!\!\!\!\mathbb{E}\big\|X^\nu_1(t)-X^\nu_2(t)\big\|^2_{F^*_{1,2,\nu}}+2(Lip
\Psi+1)^{-1}\mathbb{E}\int_0^t\big|\Psi(X^\nu_1(s))-\Psi(X^\nu_2(s))\big|_2^2ds\\
\leq
&&\!\!\!\!\!\!\!\!C_1\mathbb{E}\int_0^t\big\|X^\nu_1(s)-X^\nu_2(s)\big\|^2_{F^*_{1,2,\nu}}ds.
\end{eqnarray}
The second term in the left-hand side of \eref{eqnarray19} is
positive, thus we have
\begin{eqnarray*}
&&\mathbb{E}\big\|X^\nu_1(t)-X^\nu_2(t)\big\|^2_{F^*_{1,2,\nu}}\leq
C_1\mathbb{E}\int_0^t\big\|X^\nu_1(s)-X^\nu_2(s)\big\|^2_{F^*_{1,2,\nu}}ds.
\end{eqnarray*}
By Gronwall's inequality, we get $X^\nu_1=X^\nu_2$,
$\mathbb{P}-a.s.$, which indicates the uniquess.

\vspace{2mm} Hence the proof of Lemma 3.1 is complete.
\hspace{\fill}$\Box$
\end{proof}

\vspace{2mm} Based on Lemma 3.1, we shall now give the proof of our
main result Theorem 3.1. The idea is to prove that the sequence
$\{X^\nu\}_{\nu\in(0,1)}$ converges to the solution of \eref{eq:1}
as $\nu\to0$. The method that we use here is similar to that in
Lemma 3.1.

\vspace{2mm} {\bf Proof of Theorem 3.1 (continued)} \vspace{2mm}

First, we rewrite \eref{eq:3} as
\begin{eqnarray*}
dX^\nu(t)+(1-L)\Psi(X^\nu(t))dt=(1-\nu)\Psi(X^\nu(t))dt+B(t,X^\nu(t))dW(t).
\end{eqnarray*}

For the function $\varphi(x)=\frac{1}{2}\|x\|^2_{F^*_{1,2}}$ with $x\in F^*_{1,2}$, It\^{o}'s formula yields
\begin{eqnarray}\label{eqnarray14}
&&\frac{1}{2}\mathbb{E}\|X^\nu(t)\|^2_{F^*_{1,2}}+\int_0^t\big\langle \Psi(X^\nu(s)),X^\nu(s)\big\rangle_2ds\nonumber\\
=\!\!\!\!\!\!\!\!&&\frac{1}{2}\|x\|^2_{F^*_{1,2}}+(1-\nu)\mathbb{E}\int_0^t\big\langle \Psi(X^\nu(s)),X^\nu(s)\big\rangle_{F^*_{1,2}}ds\nonumber\\
&&+\frac{1}{2}\mathbb{E}\int_0^t\big\|B(s,
X^\nu(s))\big\|^2_{L_2(L^2(\mu),F^*_{1,2})}ds.
\end{eqnarray}
The condition \textbf{(H1)} implies
\begin{eqnarray}\label{eqna6}
\Psi(r)r\geq\tilde{\alpha}\cdot|\Psi(r)|^2,\ r\in \mathbb{R}.
\end{eqnarray}
By \eref{eqnarray14} and \eref{eqna6}, we have
\begin{eqnarray*}
&&\frac{1}{2}\mathbb{E}\|X^\nu(t)\|^2_{F^*_{1,2}}+\tilde{\alpha}\cdot\mathbb{E}\int_0^t|\Psi(X^\nu(s))|^2_2 ds\\
\leq\!\!\!\!\!\!\!\!&&\frac{1}{2}\|x\|^2_{F^*_{1,2}}+\mathbb{E}\int_0^t\|\Psi(X^\nu(s))\|_{F^*_{1,2}}\cdot\|X^\nu(s)\|_{F^*_{1,2}}ds\\
&&+\frac{1}{2}C_2\mathbb{E}\int_0^t\|X^\nu(s)\|^2_{F^*_{1,2}}ds.
\end{eqnarray*}
Since $L^2(\mu)$ is continuously embedded into $F^*_{1,2}$, Young's
inequality and the Gronwall's inequality yield that there exists a
constant $C\in(0, \infty)$ such that, for $t\in[0,T]$ and
$\nu\in(0,1)$,
\begin{eqnarray}\label{eq:9}
\mathbb{E}\|X^\nu(t)\|^2_{F^*_{1,2}}\leq C\|x\|^2_{F^*_{1,2}}.
\end{eqnarray}

In the following, we will prove the convergence of
$\{X^\nu\}_{\nu\in(0,1)}$. Applying It\^{o}'s formula to
$\|X^\nu(t)-X^{\nu'}(t)\|^2_{F^*_{1,2}}$, we get that, for all
$t\in[0,T]$,
\begin{eqnarray}\label{eqnarray15}
&&\|X^\nu(t)-X^{\nu'}(t)\|^2_{F^*_{1,2}}+2\int_0^t\big\langle(\Psi(X^\nu(s))-\Psi(X^{\nu'}(s)),X^\nu(s)-X^{\nu'}(s)\big\rangle_2 ds\nonumber\\
=\!\!\!\!\!\!\!\!&&2\int_0^t\big\langle \Psi(X^\nu(s))-\Psi(X^{\nu'}(s)),X^\nu(s)-X^{\nu'}(s)\big\rangle_{F^*_{1,2}}ds\nonumber\\
&&-2\int_0^t\big\langle \nu \Psi(X^\nu(s))-\nu' \Psi(X^{\nu'}(s)),X^\nu(s)-X^{\nu'}(s)\big\rangle_{F^*_{1,2}} ds\nonumber\\
&&+2\int_0^t\big\|B(s,X^\nu(s))-B(s,X^{\nu'}(s)\big\|^2_{L_2(L^2(\mu),
F^*_{1,2})}ds\nonumber\\
&&+2\int_0^t\big\langle
X^\nu(s)-X^{\nu'}(s),(B(s,X^\nu(s))-B(s,X^{\nu'}(s)))dW(s)\big\rangle_{F^*_{1,2}}ds.
\end{eqnarray}
The second term on the right hand side of \eref{eqnarray15} can be
dominated by
\begin{eqnarray}
&&-2\int_0^t\big\langle \nu\Psi(X^\nu(s))-\nu'\Psi(X^{\nu'}(s)),X^\nu(s)-X^{\nu'}(s)\big\rangle_{F^*_{1,2}} ds
\nonumber\\
\leq\!\!\!\!\!\!\!\!&&2C\int_0^t\big(\nu|\Psi(X^\nu(s))|_2+\nu'\big|\Psi(X^{\nu'}(s))\big|_2\big)\cdot
\|X^\nu(s)-X^{\nu'}(s)\|_{F^*_{1,2}}ds.
\end{eqnarray}
By assumption \textbf{(H1)} on $\Psi$ and \eref{eqna6}, we obtain
\begin{eqnarray}\label{eqna7}
&&2\int_0^t\big\langle(\Psi(X^\nu(s))-\Psi(X^{\nu'}(s)),X^\nu(s)-X^{\nu'}(s)\big\rangle_2ds\nonumber\\
=\!\!\!\!\!\!\!\!&&2\int_0^t\int_E\big(\Psi(X^\nu(s))-\Psi(X^{\nu'}(s))\big)\cdot\big(X^\nu(s)-X^{\nu'}(s)\big)d\mu ds\nonumber\\
\geq\!\!\!\!\!\!\!\!&& 2\int_0^t\int_E\tilde{\alpha}\big|\Psi(X^\nu(s))-\Psi(X^{\nu'}(s))\big|^2 d\mu ds\nonumber\\
=\!\!\!\!\!\!\!\!&&2\tilde{\alpha}\int_0^t\big|\Psi(X^\nu(s))-\Psi(X^{\nu'}(s))\big|^2_2ds.
\end{eqnarray}
\eref{eqnarray15}-\eref{eqna7} imply
\begin{eqnarray*}
&&\big\|X^\nu(t)-X^{\nu'}(t)\big\|^2_{F^*_{1,2}}+2\tilde{\alpha}\int_0^t\big|\Psi(X^\nu(s))-\Psi(X^{\nu'}(s))\big|^2_2ds\\
\leq\!\!\!\!\!\!\!\!&&C_1\int_0^t\big|\Psi(X^\nu(s))-\Psi(X^{\nu'}(s))\big|_2\cdot\big\|X^\nu(s)-X^{\nu'}(s)\big\|_{F^*_{1,2}}ds\\
&&+C_2\int_0^t\big(\nu|\Psi(X^\nu(s))|_2+\nu'|\Psi(X^{\nu'}(s))|_2\big)\cdot\big\|X^\nu(s)-X^{\nu'}(s)\big\|_{F^*_{1,2}}ds\\
&&+C_3\int_0^t\big\|X^\nu(s)-X^{\nu'}(s)\big\|^2_{F^*_{1,2}}ds\\
&&+2\int_0^t\big\langle
X^\nu(s)-X^{\nu'}(s),\big(B(s,X^\nu(s))-B(s,X^{\nu'}(s))\big)dW(s)\big\rangle_{F^*_{1,2}}ds.
\end{eqnarray*}
Taking expectation to both sides of the above inequality and using
 Young's and the BDG inequality for $p=1$, we obtain, for all $ t\in[0,T]$,
\begin{eqnarray*}
&&\mathbb{E}\Big[\sup_{s\in[0,t]}\big\|X^\nu(s)-X^{\nu'}(s)\big\|^2_{F^*_{1,2}}\Big]
+2\tilde{\alpha}\mathbb{E}\int_0^t\big|\Psi(X^\nu(s))-\Psi(X^{\nu'}(s))\big|_2^2ds\\
\leq\!\!\!\!\!\!\!\!&&\frac{1}{2}\mathbb{E}\Big[\sup_{s\in[0,t]}\big\|X^\nu(s)-X^{\nu'}(s)\big\|^2_{F^*_{1,2}}\Big]+
\tilde{\alpha}\mathbb{E}\int_0^t\big|\Psi(X^\nu(s))-\Psi(X^{\nu'}(s))\big|_2^2ds\\
&&+C_1\mathbb{E}\int_0^t\big\|X^\nu(s)-X^{\nu'}(s)\big\|_{F^*_{1,2}}ds
+C_2\mathbb{E}\int_0^t\big(\nu|\Psi(X^\nu(s))|^2_2+\nu'|\Psi(X^{\nu'}(s))|_2^2\big)ds.
\end{eqnarray*}
This yields
\begin{eqnarray}
&&\mathbb{E}\Big[\sup_{s\in[0,t]}\big\|X^\nu(s)-X^{\nu'}(s)\big\|^2_{F^*_{1,2}}\Big]
+2\tilde{\alpha}\mathbb{E}\int_0^t\big|\Psi(X^\nu(s))-\Psi(X^{\nu'}(s))\big|_2^2ds\nonumber\\
\leq\!\!\!\!\!\!\!\!&&C_1\mathbb{E}\int_0^t\big\|X^\nu(s)-X^{\nu'}(s)\big\|_{F^*_{1,2}}ds\nonumber\\
&&+C_2(\nu+\nu')\mathbb{E}\int_0^t\big(|\Psi(X^\nu(s))|^2_2+|\Psi(X^{\nu'}(s))|_2^2\big)ds.
\end{eqnarray}

Note that if the initial value $x\in F^*_{1,2}$ and \eref{eq:2} is
satisfied, we have \eref{eq:9}. If $x\in L^2(\mu)$, we have
\eref{equ:6}. Hence, Gronwall's inequality and Young's
inequality yields that there exists a positive constant
$C\in(0,\infty)$ which is independent of $\nu, \nu'$ such that
\begin{eqnarray}\label{eqnarray20}
&&\mathbb{E}\Big[\sup_{s\in[0,T]}\big\|X^\nu(s)-X^{\nu'}(s)\big\|^2_{F^*_{1,2}}\Big]+\mathbb{E}\int_0^T\big|\Psi(X^\nu(s))-\Psi(X^{\nu'}(s))\big|_2^2ds\nonumber\\
\leq\!\!\!\!\!\!\!\!&&C(\nu+\nu').
\end{eqnarray}
Hence, there exists an $(\mathscr{F}_t)_{t\geq0}$-adapted process
$X\!\!\in \!\!L^2(\Omega;C([0,T], F^*_{1,2}))\cap
L^2([0,T]\times\Omega; L^2(\mu))$ such that $X^\nu\rightarrow X$ in
$L^2(\Omega;C([0,T], F^*_{1,2}))$ as $\nu\rightarrow0$.

\vspace{2mm} Next, we will prove $X$ satisfies \eref{equ:3}, the
proof is similar to that in Claim 3.3.

\vspace{2mm} From above we know that
\begin{eqnarray}\label{eqnarray28}
&&X^\nu\rightarrow X\ \text{and}\ \int_0^\bullet
B(s,X^\nu(s))dW(s)\rightarrow \int_0^\bullet
B(s,X(s))dW(s),~\nu\rightarrow0
\end{eqnarray}
in $L^2\big(\Omega;C([0,T];F^*_{1,2})\big)$. \eref{equ:5} and
\eref{eqnarray28} yield that
$$\int_0^\bullet\Psi(X^\nu(s))ds,\ \nu>0,$$ converges to some element in
$L^2\big(\Omega;C([0,T],F_{1,2})\big)$ as $\nu\to0$, and from
\eref{eqnarray20} we know as $\nu\to0$,
$$\int_0^\bullet \Psi(X^\nu(s))ds\to\int_0^\bullet \Psi(X(s))ds$$
in $L^2\big(\Omega;L^2([0,T];L^2(\mu))\big)$. Hence $X$ satisfies
\eref{equ:3}. This completes the existence proof for Theorem 3.1.

\vspace{2mm} {\bf Uniqueness} \vspace{2mm}

If $X_1$ and $X_2$ are two solutions to \eref{eq:1}, we have
$\mathbb{P}-a.s.$
\begin{eqnarray}\label{eqnarray21}
&&\!\!\!\!\!\!\!\!X_1(t)-X_2(t)\!-\!L\!\int_0^t\!\big(\Psi(X_1(s))-\Psi(X_2(s))\big)ds\nonumber\\
=&&\!\!\!\!\!\!\!\!\int_0^t\!\big(B(s,X_1(s))-B(s,X_2(s))\big)dW(s),\
t\in [0,T]
\end{eqnarray}
in $\Omega\times[0,T]\times E$.

Rewrite \eref{eqnarray21} as
\begin{eqnarray}\label{eqnarray22}
&&X_1(t)-X_2(t)+(1-L)\int_0^t\big(\Psi(X_1(s))-\Psi(X_2(s))\big)ds\nonumber\\
=\!\!\!\!\!\!\!\!\!&&\int_0^t\big(\Psi(X_1(s))-\Psi(X_2(s))\big)ds+\int_0^t\big(B(s,X_1(s))-B(s,X_2(s))\big)dW(s)
\end{eqnarray}

Apply It\^{o}'s formula to $\|X_1(t)-X_2(t)\|^2_{F^*_{1,2}}$ in
$F^*_{1,2}$, it follows
\begin{eqnarray}\label{eqnarray23}
&&\big\|X_1(t)-X_2(t)\big\|^2_{F^*_{1,2}}+2\int_0^t\big\langle\Psi(X_1(s))-\Psi(X_2(s)),X_1(s)-X_2(s)\big\rangle_2ds\nonumber\\
=\!\!\!\!\!\!\!\!\!&&2\int_0^t\big\langle
\Psi(X_1(s))-\Psi(X_2(s)),X_1(s)-X_2(s)\big\rangle_{F^*_{1,2}}ds\nonumber\\
&&+2\int_0^t\big\langle
X_1(s)-X_2(s),\big(B(s,X_1(s))-B(s,X_2(s))\big)dW(s)\big\rangle_{F^*_{1,2}}\nonumber\\
&&+\int_0^t\big\|B(s,X_1(s))-B(s,X_2(s))\big\|_{L_2(L^2(\mu),F^*_{1,2})}ds.
\end{eqnarray}
Taking expectation of both sides, \eref{eqn2} and
\textbf{(H2)}\textbf{(i)} yield that
\begin{eqnarray*}
&&\mathbb{E}\big\|X_1(t)-X_2(t)\big\|^2_{F^*_{1,2}}+2\widetilde{\alpha}\mathbb{E}\int_0^t\big|\Psi(X_1(s))-\Psi(X_2(s))\big|_2^2ds\nonumber\\
\leq\!\!\!\!\!\!\!\!\!&&2\mathbb{E}\int_0^t\big\|\Psi(X_1(s))-\Psi(X_2(s))\big\|_{F^*_{1,2}}\cdot
\big\|X_1(s)-X_2(s)\big\|_{F^*_{1,2}}ds\nonumber\\
&&+C_1\mathbb{E}\int_0^t\big\|X_1(s)-X_2(s)\big\|^2_{F^*_{1,2}}ds.
\end{eqnarray*}
Using Young's inequality to the above inequality, and since
$L^2(\mu)\subset F^*_{1,2}$ continuously and densely, we obtain
\begin{eqnarray*}
&&\mathbb{E}\big\|X_1(t)-X_2(t)\big\|^2_{F^*_{1,2}}+2\widetilde{\alpha}\mathbb{E}\int_0^t\big|\Psi(X_1(s))-\Psi(X_2(s))\big|_2^2ds\nonumber\\
\leq\!\!\!\!\!\!\!\!\!&&2\widetilde{\alpha}\mathbb{E}\int_0^t\big|\Psi(X_1(s))-\Psi(X_2(s))\big|_2^2ds+
C\mathbb{E}\int_0^t\big\|X_1(s)-X_2(s)\big\|^2_{F^*_{1,2}}ds\nonumber\\
&&+C_1\mathbb{E}\int_0^t\big\|X_1(s)-X_2(s)\big\|^2_{F^*_{1,2}}ds.
\end{eqnarray*}
Therefore,
\begin{eqnarray*}
&&\mathbb{E}\big\|X_1(t)-X_2(t)\big\|^2_{F^*_{1,2}}\leq
(C+C_1)\mathbb{E}\int_0^t\big\|X_1(s)-X_2(s)\big\|^2_{F^*_{1,2}}ds.
\end{eqnarray*}
By Gronwall's lemma, we get $X_1=X_2$ $\mathbb{P}-a.s.$.
Consequently, Theorem 3.1 is completely proved. \hspace{\fill}$\Box$

\section{Some Examples}
\setcounter{equation}{0}
 \setcounter{definition}{0}

\subsection{Classical Dirichlet forms with densities}

We apply Theorem 3.1 to the Friedrichs extension of the operator
\begin{eqnarray}\label{eqnarray16}
L_0u=\Delta u+2\frac{\nabla\rho}{\rho}\cdot\nabla u,\ \ u\in
C_0^\infty(\mathbb{R}^d),
\end{eqnarray}
on $L^2(\mathbb{R}^d, \rho^2dx)$, where $dx$ denotes Lebesgue
measure and $\rho\in H^1(\mathbb{R}^d)$. Here $H^1$ is the usual
Sobolev space and $H^{-1}$ denotes its dual space.

\vspace{2mm} Clearly, $\Delta u\in L^2(\mathbb{R}^d,\rho^2dx)$,
since $u\in C_0^\infty(\mathbb{R}^d)$, and
$\frac{\nabla\rho}{\rho}\cdot\nabla u\in
L^2(\mathbb{R}^d,\rho^2dx)$, since
\begin{eqnarray*}
&&\int\big|\frac{\nabla \rho}{\rho}\cdot \nabla
u\big|^2\rho^2dx=\int|\nabla\rho|^2|\nabla u|^2dx\leq C\int
|\nabla\rho|^2 dx<\infty.
\end{eqnarray*}
 Hence $L_0$ is a well-defined linear operator from
$C_0^\infty(\mathbb{R}^d)$ to $L^2(\mathbb{R}^d,\rho^2dx)$. In
addition, by definition \eref{eqnarray16}, for all $u, v \in
C_0^\infty(\mathbb{R}^d)$, we have
\begin{eqnarray*}
\int L_0u\cdot v\rho^2dx&&\!\!\!\!\!\!\!\!=\int(\Delta
u+2\frac{\nabla\rho}{\rho}\cdot\nabla u)\ v\rho^2\ dx\\
&&\!\!\!\!\!\!\!\!=\int\Delta u\ v\rho^2\ dx+2\int\frac{\nabla\rho}{\rho}\cdot\nabla u\ v\rho^2\ dx\\
&&\!\!\!\!\!\!\!\!=\int div\nabla u\ v\rho^2\ dx+2\int\frac{\nabla\rho}{\rho}\cdot\nabla u\ v\rho^2\ dx\\
&&\!\!\!\!\!\!\!\!=-\int\nabla u\cdot\nabla(v\rho^2)\ dx+2\int\frac{\nabla\rho}{\rho}\cdot\nabla u\ v\rho^2\ dx\\
&&\!\!\!\!\!\!\!\!=-\int\nabla u\cdot\nabla v\ \rho^2\ dx-\int\nabla u\cdot 2v\rho\nabla \rho\ dx+2\int\frac{\nabla\rho}{\rho}\cdot\nabla u\ v\rho^2\ dx\\
&&\!\!\!\!\!\!\!\!=-\int\nabla u\cdot \nabla v\ \rho^2\ dx=\int
u\cdot L_0v\ \rho^2\ dx,
\end{eqnarray*}
which implies both that $L_0$ is a symmetric operator and negative
definite. Hence $(L_0,C_0^\infty(\mathbb{R}^d))$ is closable in
$L^2(\mathbb{R}^d,\rho^2dx)$ (\cite{MR}). Let $(L, D(L))$ be its
closure. According to \cite[Proposition 3.3]{MR}, we hence know that
there exists a Dirichlet form $(\mathcal {E}, D(\mathcal {E}))$ on
$L^2(\mathbb{R}^d,\rho^2dx)$, which is in fact the closure of
$$\mathcal {E}(u, v)=\int \langle \nabla u, \nabla v\rangle_{\mathbb{R}^d}\rho^2dx, \text{\ \ for all\ \ } u, v\in C_0^\infty(\mathbb{R}^d),$$
on $L^2(\mathbb{R}^d, \rho^2dx)$ with generator $(L, D(L))$. $(L,
D(L))$ is thus the Friedrichs extension of $(L_0,
C_0^\infty(\mathbb{R}^d))$ on $L^2(\mathbb{R}^d,\rho^2dx)$. As a
result, we know $T_t=e^{tL}$ is the strongly continuous contraction
sub-Markovian semigroup on $L^2(\mathbb{R}^d,\rho^2dx)$ with
generator $L$, which is negative definite and self-adjoint.

\vspace{2mm}
Consider the following equation
\begin{equation} \label{eq:11}
\left\{ \begin{aligned}
&dX(t)-L\Psi(X(t))dt=B(t,X(t))dW(t),\ \text{in}\ [0,T]\times \mathbb{R}^d,\\
&X(0)=x \in L^2(\mathbb{R}^d,\rho^2dx)\ \text{or}\ x\in F^*_{1,2}\
\text{respectively\ if\ in addition \eref{eq:2} holds.}
\end{aligned} \right.
\end{equation}
where $(L,D(L))$ is the Friedrichs extension of $(L_0,
C_0^\infty(\mathbb{R}^d))$. $\Psi$ and $B$ satisfy \noindent
\textbf{(H1)} and \noindent \textbf{(H2)} respectively, whereas the
corresponding spaces are $E:=\mathbb{R}^d$,
$L^2(\mu):=L^2(\mathbb{R}^d,\rho^2dx)$, $F_{1,2}:=D(\mathcal {E})$
and $F^*_{1,2}:=(D(\mathcal {E}))^*$.

\vspace{2mm} As a consequence, we can apply Theorem 3.1.

\subsection{General regular symmetric case}

The example in Section 4.1 is a special case of the examples in
\cite[Chapter 2]{MR}: Let $E:=U\subset \mathbb{R}^d$, $U$ open, and
$m$ a positive Radon measure on $U$ such that supp$[m]=U$. For $u,
v\in C_0^\infty(U)$, define
\begin{eqnarray}\label{eqnarray7}
\mathcal {E}(u,
v):&&\!\!\!\!\!\!\!\!\!=\sum_{i,j=1}^d\int\frac{\partial u}{\partial
x_i}\frac{\partial v}{\partial x_j}d\nu_{ij}\nonumber\\
&&\!\!\!\!\!\!\!\!\!\ \ \ +\int_{U\times
U\setminus\Delta}\big(u(x)-u(y)\big)\big(v(x)-v(y)\big)J(dx,
dy)+\int\!\! uv~dk.
\end{eqnarray}
Here $k$ is a positive Radon measure on $U$ and $J$ is a symmetric
positive Radon measure on $U\times U\setminus\Delta$, where
$\Delta:=\{(x,x)|x\in U\}$, such that for all $u\in C_0^\infty(U)$
\begin{eqnarray}\label{eqnarray8}
\int|u(x)-u(y)|^2J(dxdy)<\infty.
\end{eqnarray}
For $1\leq i,j\leq d$, $\nu_{ij}$ is a Radon measure on $U$ such
that for every $K\subset U$, $K$ compact, $\nu_{ij}(K)=\nu_{ji}(K)$
and $\sum_{i,j=1}^d\xi_i\xi_j\nu_{ij}(K)\geq 0$ for all
$\xi_i,\cdot\cdot\cdot , \xi_d\in \mathbb{R}^d$.

\vspace{2mm} Then $(\mathcal {E}, C_0^\infty(U))$ is a densely
defined symmetric positive definite bilinear form on $L^2(U;m)$.

\vspace{2mm} Suppose that $(\mathcal {E}, C_0^\infty(U))$ is
closable on $L^2(U; m)$ and let $(\mathcal {E}, D(\mathcal {E}))$ be
its closure, then $(\mathcal {E}, D(\mathcal {E}))$ is a symmetric
Dirichlet form. Hence by \cite{MR} we know there exists a
self-adjoint negative definite linear operator $(L, D(L))$ on
$L^2(U; m)$ defined by
\begin{eqnarray}\label{eqnarray26}
D(L):=\{u\in D(\mathcal {E})|\ \exists\ Lu\in L^2(U; m),~~ s.t.
~~\mathcal {E}(u,v)=(-Lu,v), \forall v\in D(\mathcal {E})\}.
\end{eqnarray}
Hence $(L, D(L))$ is the generator of a sub-Markovian strongly
continuous contraction semigroup $(T_t)_{t>0}$ on $L^2(U; m)$ given
by
\begin{eqnarray}\label{eqnarray27}
T_t:=e^{tL},\ t>0.
\end{eqnarray}
\vspace{2mm} Consider the following equation
\begin{equation} \label{eq:12}
\left\{ \begin{aligned}
&dX(t)-L\Psi(X(t))dt=B(t,X(t))dW(t),\ \text{in}\ [0,T]\times U,\\
&X(0)=x\in L^2(U, m)\ \text{or}\ x\in F^*_{1,2}\ \text{respectively\
if\ in addition \eref{eq:2} holds.}
\end{aligned} \right.
\end{equation}
where $(L, D(L))$ (defined in \eref{eqnarray26}) is the generator of
a sub-Markovian strongly continuous contraction semigroup
$(T_t)_{t>0}$ on $L^2(U; m)$. $\Psi$ and $B$ satisfy \noindent
\textbf{(H1)} and \noindent \textbf{(H2)} respectively, whereas the
corresponding spaces are $E:=U$, $L^2(\mu):=L^2(U, m)$,
$F_{1,2}:=D(\mathcal {E})$ and $F^*_{1,2}:=(D(\mathcal {E}))^*$.

\vspace{2mm} Consequently, from Theorem 3.1 we know there exists a
unique strong solution to \eref{eq:12}.

\medskip
\noindent \textbf{Remark 4.1:}

\vspace{2mm} \textbf{(i)} Our result thus in particular applies to
the case when $L$ is the fractional Laplace operator
$$L:=-(-\Delta)^\alpha,\ \alpha\in(0,1],$$
since it is just a special case of the above (see \cite[Chapter
2]{MR}).

\vspace{2mm} \textbf{(ii)} Using Dirichlet form theory on fractals,
Theorem 3.1 also applies to the case when $L$ is the Laplace
operator on a fractal, and the corresponding state space $E$ is this
fractal, (see \cite{JK, K2} for details).

\medskip
\noindent \textbf{Acknowledgements}

The second and third authors acknowledge the financial support from
NSF of China and the project funded by the Priority Academic Program
Development of Jiangsu Higher Education Institutions. Support by the
DFG' through CRC 701 is also gratefully acknowledged.

\end{document}